\documentclass[a4paper, 12pt, twoside, onecolumn ]{article}
\usepackage[utf8]{inputenc}
\usepackage{amsmath, amssymb, amsthm,mathrsfs}
\usepackage{enumerate} 
\usepackage{graphicx}
\usepackage[top=1in, bottom=1.3in, left=1.3in, right=1in]{geometry}

\usepackage{lmodern}
\usepackage[all]{xy}
\usepackage{tikz-cd} 
 \usepackage{float}
 \usepackage{hyperref}
 \usepackage{accents}
 \usepackage{verbatim}
\date{}
\usepackage{xcolor}

\newtheorem{defn}{\bf Definition}[section] 
\newtheorem{exam}[defn]{\bf Example}

\newtheorem{prop}[defn]{\bf Proposition}
\newtheorem{lem}[defn]{\bf Lemma}
\newtheorem{thm}[defn]{\bf Theorem} 
\newtheorem{cor}[defn]{\bf Corollary}
\newtheorem{rmk}{\bf Remark}

\numberwithin{equation}{section}
\usepackage[pagewise]{lineno}

\newcommand{\footremember}[2]{%
    \footnote{#2}
    \newcounter{#1}
    \setcounter{#1}{\value{footnote}}%
}
\newcommand{\footrecall}[1]{%
    \footnotemark[\value{#1}]%
} 
\title{ A Linearized L1-Galerkin FEM for Non-smooth Solutions of Kirchhoff type  Quasilinear Time-fractional Integro-differential Equation }
\author{%
  Lalit Kumar \footremember{alley}{Department of Mathematics, Indian Institute of Technology Bombay, Mumbai-400076, India. lalitccc528@gmail.com}%
  \and Sivaji Ganesh Sista\footremember{trailer}{Department of Mathematics, Indian Institute of Technology Bombay, Mumbai-400076, India. siva@math.iitb.ac.in}%
  \and Konijeti Sreenadh\footrecall{alley} \footnote{Department of Mathematics, Indian Institute of Technology Delhi, New Delhi-110016, India. sreenadh@math.iitd.ac.in}%
  }

\date{}

\begin{document}

\maketitle

\begin{abstract}
\noindent In this article, we study the semi discrete and fully discrete formulations for a  Kirchhoff type quasilinear integro-differential equation \eqref{16-9-22-1} involving time-fractional  derivative of order $\alpha \in (0,1)$. For the semi discrete formulation of  the equation  \eqref{16-9-22-1}, we discretize the space domain using a conforming FEM and keep the time variable continuous. We modify the standard Ritz-Volterra projection  operator to carry out error analysis for the semi discrete formulation of the considered equation. In general, solutions of the time-fractional partial differential equations (PDEs)  have a weak singularity near time $t=0$. Taking this singularity into account, we develop a new linearized fully discrete numerical scheme for the equation \eqref{16-9-22-1} on a graded mesh in time. We derive a priori bounds on the solution of this fully discrete numerical scheme using a new weighted $H^{1}(\Omega)$ norm. We prove that the developed numerical scheme has an accuracy rate of $O(P^{-1}+N^{-(2-\alpha)})$ in $L^{\infty}(0,T;L^{2}(\Omega))$ as well as in  $L^{\infty}(0,T;H^{1}_{0}(\Omega))$, where $P$ and $N$ are degrees of freedom in the space and time directions respectively. The robustness and efficiency of the proposed numerical scheme are demonstrated by some numerical examples.
\end{abstract}



\section{Introduction}
The purpose of this article is to study the following time-fractional integro-differential equation with Kirchhoff type principal operator. Find $u:=u(x,t)$ such that 
\begin{equation} \tag{$\mathcal{D}^{\alpha}$} \label{16-9-22-1}
	^{C}D^{\alpha}_{t}u-M\left(\int_{\Omega}|\nabla u|^{2}~dx\right)\Delta u=f(x,t)+\int_{0}^{t}b(x,t,s)u(s)~ds ~\text{in}~ \Omega \times (0,T],
\end{equation}
with initial and boundary conditions
\begin{align*}
	u(x,0)&=u_{0}(x) \quad \text{in} ~\Omega,\\
	u(x,t)&=0 \quad \text{on} ~\partial \Omega \times [0,T],
\end{align*}
where the domain  $\Omega$ is a convex and bounded subset of $\mathbb{R}^{d}~(d\geq 1)$ with smooth boundary $\partial \Omega$ and $[0,T]$ is a fixed finite time interval. Nonlocal diffusion coefficient $M: (0,\infty) \rightarrow \mathbb{R}$ is a Lipschitz continuous function,  initial data $ u_{0}$, and  source term $f$ are known functions. The  memory operator $b(x,t,s)$ is defined  as a second order partial differential operator with smooth coefficients in  each variable $(x,t,s) \in \bar{\Omega}\times [0,T]\times [0,T]$
\begin{equation}
	b(x,t,s)u(s)=-\nabla \cdot (b_{2}(x,t,s)\nabla u(s))+\nabla \cdot (b_{1}(x,t,s)u(s))+b_{0}(x,t,s)u(s),
\end{equation}
with $b_{2}:\bar{\Omega}\times[0,T]\times [0,T] \rightarrow \mathbb{R}^{d\times d}$ is a symmetric and positive definite matrix with entries $[b_{2}^{ij}(x,t,s)]~(1\leq i,j \leq d)$, $b_{1}:\bar{\Omega} \times[0,T]\times [0,T] \rightarrow \mathbb{R}^{d}$ is a vector with entries $[b_{1}^{j}(x,t,s)]~(1\leq j \leq d)$, and  $b_{0}: \bar{\Omega}\times [0,T]\times [0,T] \rightarrow \mathbb{R}$ is a scalar function. The notation 
$^{C}D_{t}^{\alpha}u$ in the equation \eqref{16-9-22-1}  is the fractional time derivative of order $\alpha$ in the Caputo sense
\begin{equation}\label{LS1}
	^{C}D_{t}^{\alpha}u:=\frac{1}{\Gamma(1-\alpha)} \int_{0}^{t}\frac{1}{(t-s)^{\alpha}}\frac{\partial u }{\partial s}(s)~ds,
\end{equation}
where $\Gamma(\cdot)$ denotes the gamma function.
\par Time-fractional partial differential  equations arise naturally in the studies of physical processes like diffusion in a medium with fractal geometry, flow in highly heterogeneous aquifers, single-molecular protein dynamics  and have attracted the attention of researchers \cite{berkowitz2002physical,kou2008stochastic,nigmatullin1986realization}. There are some physical and biological models in which the unknown quantity $u$ depends on the whole space domain through the nonlocal diffusion coefficient $M\left(\int_{\Omega}|\nabla u|^{2}dx\right)$, for example  Kirchhoff-Carrier model for vibrations of an elastic string, Kirchhoff equations with a magnetic field, diffusion of a bacteria in a jar \cite{carrier1945non, mingqi2018critical, mingqi2018nonlocal}. There are some models in which history of the phenomena is captured  by the memory term $\int_{0}^{t}b(x,t,s)u(s)~ds$, for example viscoelastic forces in non-Newtonian fluids, heat conduction materials with memory. These types of models are driven by  integro-differential equations \cite{ferreira2007memory,mahata2021finite, miller1978integrodifferential, olmstead1986bifurcation}.
\par For the integer time derivative $(\alpha=1)$, authors in  \cite{kumar2020finite} proposed a linearized backward Euler-Galerkin FEM  with an accuracy rate of  $O(h+k)$, where $h$ and $k$ are the discretization parameters in the space and time directions respectively.
The current  work is in continuation of our  work  \cite{lalit}, where we have  proposed a linearized backward Euler-Galerkin FEM based on the L1 scheme  with convergence rate of  $O(h+k^{2-\alpha})$ in $L^{\infty}(0,T;L^{2}(\Omega))$ norm. This convergence analysis in \cite{lalit} was carried out for  sufficiently smooth solutions ($u \in C^{4}([0,T];H^{2}(\Omega)\cap H^{1}_{0}(\Omega))$ ) of the equation \eqref{16-9-22-1}.
\par The  solutions of the equation  \eqref{16-9-22-1} may not have required smoothness even if the given data is very smooth. This situation arises in the equation \eqref{16-9-22-1} for  $\Omega=[0,\pi],~T=1,~M=1,~f=0,~b=0$ and $u_{0}(x)=\sin x$. In this case $|u_{t}(x,t)|$ blows up as $t$ goes to zero for $a.e.~x \in [0,\pi]$ \cite[Example 2.2]{stynes2017error}. Thus, the proposed numerical scheme in \cite{lalit} may  have a loss of accuracy  for  non-smooth solutions of the equation\eqref{16-9-22-1} (Section \ref{15-10-22-4}, Table \ref{table2}). In this article, we recover this loss of accuracy by designing a robust and efficient new  linearized numerical scheme on a graded  mesh in time.
\par In this work, we study semi discrete formulation corresponding to the equation \eqref{16-9-22-1}   by discretizing the space domain through a conforming FEM \cite{thomee2007galerkin} while keeping the time variable continuous. We introduce a modified Ritz-Volterra  projection operator and derive its best approximation properties. These best approximation properties  are similar to the best approximation properties of  a standard Ritz-Volterra projection operator \cite{cannon1988non}. The best approximation properties of the modified Ritz-Volterra projection operator play a key role in  deriving error estimates of the semi discrete formulation  of the equation \eqref{16-9-22-1}.
\par We also propose a  fully discrete numerical scheme for the equation \eqref{16-9-22-1} by dividing  the  time direction into nonuniform time steps using a graded mesh. In this scheme Caputo fractional derivative is approximated by the L1 scheme based on the backward Euler difference formula \cite{lin2007finite}. We present a  new linearization technique \eqref{linearization} on a graded mesh in time for the nonlinear Kirchhoff term $M(\int_{\Omega}|\nabla u|^{2}dx)$.  To remain in the linearization setup, we use a  modified trapezoidal rule for the memory term present in the equation \eqref{16-9-22-1}.
\par We derive a priori bounds on the solution of the fully discrete numerical scheme using a new weighted  $H^{1}(\Omega)$ norm. The well-posedness of the fully discrete numerical scheme is established by a consequence of the Br\"{o}uwer fixed point theorem \cite{thomee2007galerkin}. Under limited smoothness of solutions of the equation  \eqref{16-9-22-1}, we derive local truncation errors for the approximation of nonlinearity as well as  memory term present in the equation \eqref{16-9-22-1}.  We prove that the sequence of solution of the fully discrete numerical scheme converges to the exact solution of the equation \eqref{16-9-22-1} with an accuracy rate of $O(P^{-1}+N^{-(2-\alpha)})$ in the $L^{\infty}(0,T;L^{2}(\Omega))$ as well as in the $L^{\infty}(0,T;H^{1}_{0}(\Omega))$ norm. The sharpness of  theoretical results  is illustrated by some numerical examples.\\
\noindent \textbf{The novelty of this work is summarized below.}
\begin{enumerate}
\item \textbf{Modified Ritz-Volterra projection operator \eqref{2.5}:} We modify the standard  Ritz-Volterra projection operator in such a way that it simplifies the complexities arising from the nonlocal nonlinear diffusion coefficient $M(\int_{\Omega}|\nabla u|^{2}dx)$.
\item \textbf{Graded mesh in time \eqref{SU4}:} Solutions of the linear  time-fractional PDEs have a weak singularity near initial time $t=0$ \cite[Theorem 2.1]{stynes2017error}. This phenomena causes some loss of accuracy (Section \ref{15-10-22-4}, Table \ref{table2}) for the numerical scheme given in \cite{lalit}. We recover this loss  of accuracy by constructing a numerical scheme on a  graded mesh in time.
\item \textbf{Linearization technique \eqref{linearization}:} 
The presence of the nonlocal nonlinear diffusion coefficient in the equation  \eqref{16-9-22-1}  destroys the sparsity of the Jacobian of the Newton-Raphson method \cite{gudi2012finite}. Thus, the Newton-Raphson   method becomes  highly expensive for solving the nonlinear system arising from the fully discrete numerical scheme of the equation \eqref{16-9-22-1}. We  minimize this cost by designing  a  new linearization technique \eqref{linearization} on a  graded mesh in time  for the nonlocal nonlinear diffusion coefficient.  We remark that this linearization technique is new in the literature which can be used for any quasilinear PDE.
 \item \textbf{Weighted $H^{1}(\Omega)$ norm \eqref{new norm}:}  Due to the approximation of the memory term on a  graded mesh in time, we cannot apply the discrete fractional Gr\"{o}nwall's inequalities \cite{jin2018numerical,ren2021sharp} to establish the global rate of convergence. We derive the  global rate of convergence for the proposed numerical scheme using a new weighted $H^{1}(\Omega)$ norm \eqref{new norm} and  standard discrete Gr\"{o}nwall's inequality.
\end{enumerate}
\noindent \textbf{Rest of  the paper is organized  as follows}. In Section \ref{15-10-22-1}, we setup some notations and assumptions on given data. We also prove well-posedness and regularity results on the  weak solution of the equation \eqref{16-9-22-1}. In Section \ref{15-10-22-2}, we introduce a modified Ritz-Volterra projection operator and derive  error estimates for the semi discrete formulation of the equation \eqref{16-9-22-1}. In Section \ref{15-10-22-3}, we construct a new linearized L1 Galerkin finite element numerical scheme on a graded mesh in time  and prove the global rate of convergence. In Section \ref{15-10-22-4}, these convergence estimates are confirmed by testing two  examples through  MATLAB  programming. This work is concluded in Section \ref{15-10-22-5}.

\section{Existence, uniqueness and regularity of the  weak solution  of  the equation  \eqref{16-9-22-1} }\label{15-10-22-1}
Let $L^{2}(\Omega)$ be the Lebesgue space of square integrable functions with the norm $\|\cdot\|$ induced by the inner product $(\cdot,\cdot)$ and $H^{m}(\Omega),~m\geq 1$ be the standard Sobolev spaces with the norm $\|\cdot\|_{m}$. We denote  the space of functions in  $H^{m}(\Omega)$ having zero trace on the boundary $\partial \Omega$ by $H^{m}_{0}(\Omega)$. Further for some space $X$,~ $L^{2}\left(0,T;X\right)$ and $ L^{\infty}\left(0,T;X\right)$ be the spaces with the norm defined by 
\begin{equation*}
    \|u\|_{L^{2}\left(0,T;X\right)}^{2}:=\int_{0}^{T}\|u(t)\|_{X}^{2}~dt,\;\;
\text{and}\; \;
  \|u\|_{L^{\infty}\left(0,T;X\right)} :=\text{ess}\sup_{0\leq t\leq T}\|u(t)\|_{X},
\end{equation*}
respectively. We also define a space $L^{2}_{\alpha}(0,T;X)$ with the following norm \cite{li2018some}
\begin{equation*}
    \|u\|^{2}_{L^{2}_{\alpha}(0,T;X)}:=\sup_{t \in (0,T)}\left(\frac{1}{\Gamma(\alpha)}\int_{0}^{t}(t-s)^{\alpha-1}\|u(s)\|_{X}^{2}~ds\right).
\end{equation*}
For any two quantities $a$ and $b$, the notation $a \lesssim b$ means that there exists a generic positive constant $C_{0}$ such that $a\leq C_{0} b$. Here constant $C_{0}$ may depends on  given data but  independent of discretization parameters.
\par Throughout this work, we  make the following assumptions on  data:\\
\textbf{(A1)}~ Initial data  $u_{0} \in H^{2}(\Omega)\cap H^{1}_{0}(\Omega)$ and source term $f \in L^{\infty}\left(0,T;L^{2}(\Omega)\right)$.
\textbf{(A2)}~ There exists a positive constant $m_{0}>0$  such that 
\begin{equation}
    M(s)\geq m_{0}>0~\text{for all}~s~ \in~ (0,\infty)~\text{and}~\left(m_{0}-4L_{M}K^{2}\right)>0,
\end{equation}
where  $L_{M}$  is a  Lipschitz constant and $ K:= \| u_{0}\|_{1}+\|f\|_{L^{\infty}(0,T;L^{2}(\Omega))}$.
\par The weak formulation corresponding to the equation \eqref{16-9-22-1} is to seek\\
$u \in L^{2}(0,T;H^{2}(\Omega)\cap H^{1}_{0}(\Omega))$ with $^{C}D^{\alpha}_{t}u \in L^{2}(0,T;L^{2}(\Omega))$ such that  following equations hold for all $v$ in $H^{1}_{0}(\Omega)$ and $a.e.~ t \in (0,T)$
\begin{equation}\label{nonlinear weak problem}
 \begin{aligned}
     \left(^{C}D_{t}^{\alpha}u,v\right)+M\left(\|\nabla u\|^{2}\right)(\nabla u,\nabla v)&=(f,v)+\int_{0}^{t}B(t,s,u(s),v)~ds,\\
     u(x,0)&=u_{0}(x)~~\text{in}~~\Omega.
 \end{aligned}
 \end{equation}
The function $B(t,s,u(s),v)$ is defined  for all $t,s$ in $[0,T]$ and for all $u$, $v$ in $H^{1}_{0}(\Omega)$ as 
\begin{equation*}
    B(t,s,u(s),v):= (b_{2}(x,t,s)\nabla u(s),\nabla v)-(b_{1}(x,t,s)u(s),\nabla v)+(b_{0}(x,t,s)u(s),v).
\end{equation*}
\noindent The Poincar\'{e} inequality and assumptions on coefficients of the memory operator $b(x,t,s)$ imply that there exists a positive constant $B_{0}$ such that for all $(t,s)$ in $[0,T]\times [0,T]$ and for all $u,v$ in $H^{1}_{0}(\Omega)$ the following inequality holds 
\begin{equation}\label{2.3}
|B(t,s,u(s),v)|\leq B_{0}\|\nabla u(s)\|\|\nabla v\|. 
\end{equation} 
\begin{thm}\cite{lalit}\label{a priori bounds } Suppose that $(A1)$ and $(A2)$ hold. Then there exists a unique solution to the problem \eqref{nonlinear weak problem} which  satisfies the following a priori bounds
\begin{align}
	\|u\|^{2}_{L^{\infty}\left(0,T;L^{2}(\Omega)\right)}+\|u\|^{2}_{L^{2}_{\alpha}(0,T;H^{1}_{0}(\Omega))} &\lesssim \left(\| u_{0}\|^{2}+\|f\|^{2}_{L^{\infty}(0,T;L^{2}(\Omega))}\right).\label{AMP1as1}\\
\|u\|^{2}_{L^{\infty}\left(0,T;H^{1}_{0}(\Omega)\right)}+\|u\|^{2}_{L^{2}_{\alpha}(0,T;H^{2}(\Omega))}& \lesssim \left(\| \nabla u_{0}\|^{2}+\|f\|^{2}_{L^{\infty}(0,T;L^{2}(\Omega))}\right).\label{AMP1as}
\end{align}
\end{thm}
\noindent \textbf{Regularity of the weak solution:} Authors in  \cite{huang2020optimal,kopteva2019error} studied  the following linear  time-fractional initial boundary value problem     
 \begin{equation} \label{22-8-22-1}
\begin{aligned}
^{C}D^{\alpha}_{t}u-\Delta u&=f(x,t)\quad \text{in}~ \Omega \times (0,T],\\
u(x,0)&=u_{0}(x) \quad \text{in} ~\Omega,\\
u&=0 \quad \text{in}~\partial \Omega \times (0,T].
\end{aligned}
\end{equation}
\noindent Let $\{(\lambda_{i},\phi_{i}), i=1,2,3,\dots\}$ be the eigenvalues and eigenfunctions for the boundary value problem  $-\Delta \phi_{i}=\lambda_{i}\phi_{i}~~ \text{in}~~ \Omega $ with homogeneous Dirichlet boundary condition. The eigenfunctions are normalised by taking $\|\phi_{i}\|=1$ for all $i$. From the theory of sectorial operator \cite{henry2006geometric} the fractional power $\Delta^{\mu}$ of the operator $\Delta$ is defined for each $\mu \in \mathbb{R}$ with domain 
\begin{equation*}
D(\Delta^{\mu})=\{g \in  L^{2}(\Omega);~ \sum_{j=1}^{\infty}\lambda_{j}^{2\mu}~|(g,\phi_{j})|^{2} < \infty\},
\end{equation*}
and the norm is given by 
\begin{equation}\label{22-8-22-1-1-1}
    \|g\|^{2}_{\Delta^{\mu}}= \sum_{j=1}^{\infty}\lambda_{j}^{2\mu}~|(g,\phi_{j})|^{2}.
\end{equation}
\noindent For example, $D(\Delta^{1/2})=H^{1}_{0}(\Omega)$ and $D(\Delta)=H^{2}(\Omega)\cap H^{1}_{0}(\Omega)$. They have proved the following a priori  bounds on the solution of  the problem \eqref{22-8-22-1}.
\begin{lem} \cite{huang2020optimal}
Let $q \in \mathbb{N}\cup \{0\}$. Assume that $u_{0}\in D(\Delta^{q+2})$ and $\frac{\partial^{l}f}{\partial t^{l}}(\cdot,t)\in D(\Delta^{q})$ for $t\in (0,T]$ and that $\|u_{0}\|_{\Delta^{q+2}}+\|\frac{\partial^{l}f}{\partial t^{l}}(\cdot,t)\|_{\Delta^{q+1}} \lesssim C_{0}$ for $l=1,2$. Then the solution of the problem \eqref{22-8-22-1} satisfies 
\begin{align}
\|u(\cdot,t)\|_{q}& \lesssim C_{0},~\forall~ t \in (0,T]\label{14-9-22-1}\\
\|\frac{\partial^{l}u}{\partial t^{l}}(\cdot,t)\|_{q} &\lesssim C_{0}(1+t^{\alpha-l})~~l=0,1,2~\forall ~t \in (0,T]\label{14-9-22-2}.
\end{align}
\end{lem}
\noindent Further, authors in \cite{jin2019subdiffusion} extended the problem \eqref{22-8-22-1} to the following subdiffusion problem with time dependent coefficient
\begin{equation} \label{22-8-22-1-1}
\begin{aligned}
^{C}D^{\alpha}_{t}u-\nabla \cdot(a(x,t)\nabla u)&=f(x,t)\quad \text{in}~ \Omega \times (0,T],\\
u(x,0)&=u_{0}(x) \quad \text{in} ~\Omega,\\
u&=0 \quad \text{in}~\partial \Omega \times (0,T],
\end{aligned}
\end{equation}
where $a(x,t)\in \mathbb{R}^{d \times d}$ is a symmetric  matrix valued diffusion coefficient such that for some constant $\lambda\geq 1$ 
\begin{equation*}
\begin{aligned}
\lambda^{-1}|\xi|^{2} \leq a(x,t)\xi\cdot \xi \leq \lambda |\xi|^{2}~~~\forall\xi\in \mathbb{R}^{d \times d}~~~\forall~(x,t)\in\Omega \times (0,T],\\
|\partial_{t}a(x,t)|+|\nabla a (x,t)|+|\nabla \partial_{t}a(x,t)| \lesssim C_{0}~~~\forall~(x,t)\in\Omega \times (0,T].
\end{aligned}
\end{equation*}
They have established the following regularity results for the solution of \eqref{22-8-22-1-1} using perturbation argument.
\begin{lem}\cite{jin2019subdiffusion} Suppose $u_{0}\in H^{2}(\Omega) \cap H^{1}_{0}(\Omega)$ and $f=0$ in \eqref{22-8-22-1-1}. Then the solution $u$ of \eqref{22-8-22-1-1} undergoes 
\begin{align}
\|u(t)\|_{2} \lesssim \|u_{0}\|_{2}~~~\forall~t\in [0,T],\label{18-12-22-1} \\
\|\frac{\partial u}{\partial t}\| \lesssim t^{\alpha-1}\|u_{0}\|_{2}~~~\forall~t\in (0,T].\label{18-12-22-1-2}
\end{align} 
\end{lem} 
\noindent In this work, the problem \eqref{16-9-22-1} contains a Kirchhoff type nonlinearity due to which we cannot apply techniques used in \cite{jin2019subdiffusion, kopteva2019error, sakamoto2011initial} to derive regularity properties \eqref{18-12-22-1}-\eqref{18-12-22-1-2} of the   solution of \eqref{16-9-22-1}. Here, we solve this problem partially and find it difficult to prove remaining part (temporal regularity of $u$) at this moment. we will address it into another work. 
\begin{thm}
 Suppose that (A1) and (A2) hold. Then the solution $u$ of \eqref{16-9-22-1} satisfies following estimates 
\begin{align}
\|u(t)\|_{2} &\lesssim C~~~\forall~t\in [0,T], \label{18-12-22-2}\\
\|u(t)-u(\tau)\| &\lesssim |t-\tau|^{\alpha}~~~\forall~t,\tau\in [0,T].\label{18-12-22-3}
\end{align}
\end{thm}
\begin{proof} First we consider the case $f=b=0$ in \eqref{16-9-22-1}, then the weak formulation of the problem \eqref{16-9-22-1}  is given by 
\begin{equation} \label{18-12-22-4}
\begin{aligned}
\left(^{C}D^{\alpha}_{t}u,v\right)+M\left(\|\nabla u\|^{2}\right)(\nabla u,\nabla v)&=0\quad \forall~v \in H^{1}_{0}(\Omega), ~a.e.~t\in (0,T],\\
u(x,0)&=u_{0}(x) \quad \text{in} ~\Omega.
\end{aligned}
\end{equation}
By putting $u$ as a Fourier series expansion i.e., $u(x,t)=\sum_{i=1}^{\infty}\alpha_{i}(t)\phi_{i}(x)$ in \eqref{18-12-22-4}, we get 
\begin{equation}\label{18-12-22-5}
\begin{aligned}
^{C}D^{\alpha}_{t}\alpha_{i}(t)&=-\lambda_{i}M\left(\|\nabla u(t)\|^{2}\right)\alpha_{i}(t)~~\forall~i=1,2,3,\dots\\
\alpha_{i}(0)&=(u_{0},\phi_{i})~~\forall~i=1,2,3,\dots.
\end{aligned}
\end{equation}
The equation \eqref{18-12-22-5} is converted into following integral equation 
\begin{equation*}\label{18-12-22-6}
\begin{aligned}
\alpha_{i}(t)&=\alpha_{i}(0)-\lambda_{i}\int_{0}^{t}\frac{(t-s)^{\alpha-1}}{\Gamma(\alpha)}M\left(\|\nabla u(s)\|^{2}\right)\alpha_{i}(s)~ds~~\forall~i=1,2,3,\dots.
\end{aligned}
\end{equation*}
Take modulus on both sides and apply (A2) together with generalised Gr\"{o}nwall's inequality \cite{almeida2017gronwall} to deduce
\begin{equation*}\label{18-12-22-7}
|\alpha_{i}(t)|\leq |\alpha_{i}(0)|E_{\alpha}[\lambda_{i}t^{\alpha}] \leq |(u_{0},\phi_{i})|~~\forall~i=1,2,3,\dots,
\end{equation*}
where $E_{\alpha}(\cdot)$ is the Mittag-Leffler function defined in \cite{podlubny1998fractional}.
Further, by definition \eqref{22-8-22-1-1-1} we obtain 
\begin{equation}\label{18-12-22-8}
\|u\|^{2}=\sum_{i=1}^{\infty}|\alpha_{i}(t)|^{2} \lesssim \sum_{i=1}^{\infty} |(u_{0},\phi_{i})|^{2} =\|u_{0}\|^{2},
\end{equation}
\begin{equation}\label{18-12-22-9}
\|\nabla u\|^{2}=\sum_{i=1}^{\infty}\lambda_{i}|\alpha_{i}(t)|^{2} \lesssim \sum_{i=1}^{\infty} \lambda_{i}|(u_{0},\phi_{i})|^{2} =\|u_{0}\|^{2}_{1},
\end{equation}
\begin{equation}\label{18-12-22-10}
\|\Delta u\|^{2}=\sum_{i=1}^{\infty}\lambda_{i}^{2}|\alpha_{i}(t)|^{2} \lesssim \sum_{i=1}^{\infty} \lambda_{i}^{2} |(u_{0},\phi_{i})|^{2} =\|u_{0}\|^{2}_{2}.
\end{equation}
By combining the estimates \eqref{18-12-22-8}-\eqref{18-12-22-10}, we get $\|u(t)\|_{2} \lesssim \|u_{0}\|_{2}~\forall~t\in [0,T]$. Now consider the equation \eqref{16-9-22-1} with $f=b=0$, we have 
\begin{equation*} \label{18-12-22-11}
\begin{aligned}
	^{C}D^{\alpha}_{t}u&=M\left(\int_{\Omega}|\nabla u|^{2}~dx\right)\Delta u~\text{in}~ \Omega \times (0,T],\\
	u(x,0)&=u_{0}(x) \quad \text{in} ~\Omega,\\
	u(x,t)&=0 \quad \text{on} ~\partial \Omega \times [0,T].
 \end{aligned}
\end{equation*}
Then $u$ satisfies the following integral equation 
\begin{equation*}\label{18-12-22-12}
\begin{aligned}
u(t)&=u_{0}+\int_{0}^{t}\frac{(t-s)^{\alpha-1}}{\Gamma(\alpha)}M\left(\|\nabla u(s)\|^{2}\right)\Delta u(s)~ds.
\end{aligned}
\end{equation*}
Estimate \eqref{18-12-22-10} and (A2) imply 
\begin{equation*}\label{18-12-22-13}
\begin{aligned}
\|u(t)\|\lesssim \left(1+t^{\alpha}\right)\|u_{0}\|_{2}.
\end{aligned}
\end{equation*}
To prove the estimate \eqref{18-12-22-3}, let us take $\tau<t$ then 
\begin{equation*}
\begin{aligned}
u(t)-u(\tau)&=\int_{0}^{t}\frac{(t-s)^{\alpha-1}}{\Gamma(\alpha)}M\left(\|\nabla u(s)\|^{2}\right)\Delta u(s)~ds\\
&-\int_{0}^{\tau}\frac{(\tau-s)^{\alpha-1}}{\Gamma(\alpha)}M\left(\|\nabla u(s)\|^{2}\right)\Delta u(s)~ds,
\end{aligned}
\end{equation*}
which can also be rewritten as 
\begin{equation*}
\begin{aligned}
u(t)-u(\tau)&=\frac{1}{\Gamma(\alpha)}\int_{0}^{\tau}\left((t-s)^{\alpha-1}-(\tau-s)^{\alpha-1}\right)M\left(\|\nabla u(s)\|^{2}\right)\Delta u(s)~ds\\
&+\frac{1}{\Gamma(\alpha)}\int_{\tau}^{t}(t-s)^{\alpha-1}M\left(\|\nabla u(s)\|^{2}\right)\Delta u(s)~ds.
\end{aligned}
\end{equation*}
Using estimates \eqref{18-12-22-8}-\eqref{18-12-22-10} along with (A2), we arrive at 
\begin{equation*}
\begin{aligned}
\|u(t)-u(\tau)\|&\lesssim \frac{1}{\Gamma(\alpha)}\int_{0}^{\tau}\left|(t-s)^{\alpha-1}-(\tau-s)^{\alpha-1}\right|~ds\\
&+\frac{1}{\Gamma(\alpha)}\int_{\tau}^{t}(t-s)^{\alpha-1}ds.
\end{aligned}
\end{equation*}
For $\tau<t$, $(t-s)^{\alpha-1}-(\tau-s)^{\alpha-1}<0$. Therefore,
\begin{equation*}
\begin{aligned}
\|u(t)-u(\tau)\|&\lesssim \frac{1}{\Gamma(\alpha)}\int_{0}^{\tau}\left((\tau-s)^{\alpha-1}-(t-s)^{\alpha-1}\right)~ds\\
&+\frac{1}{\Gamma(\alpha)}\int_{\tau}^{t}(t-s)^{\alpha-1}~ds.\\
&\lesssim (t-\tau)^{\alpha}+\tau^{\alpha}-(t)^{\alpha}+(t-\tau)^{\alpha} \lesssim (t-\tau)^{\alpha}.
\end{aligned}
\end{equation*}
Using similar arguments one can prove $\|u(\tau)-u(t)\|\lesssim (\tau-t)^{\alpha}$. Hence proof of estimate \eqref{18-12-22-3} follows for $f=b=0$. For the case $f\neq 0$ and $b\neq 0$, the estimate \eqref{18-12-22-3} can be proved easily using similar steps. 
\end{proof}
\noindent In this paper, we derive semi discrete and fully discrete error estimates under the assumption that the solution $u$ of the equation \eqref{16-9-22-1} satisfies temporal regularity \eqref{14-9-22-2}.

\section{Semi discrete formulation of the  equation \eqref{16-9-22-1} and its error estimates}\label{15-10-22-2}
Let $\mathbb{T}_{h}$ be a shape regular and quasi-uniform triangulation of the domain $\Omega$ into d-simplexes  denoted by $T_{h}$, called finite elements. Then over the triangulation $\mathbb{T}_{h}$, we define a continuous piecewise linear finite element space $X_{h}$ \cite{thomee2007galerkin} given by 
\begin{equation*}
    X_{h}=\left\{v_{h}\in H^{1}_{0}(\Omega): v_{h}~\text{is a linear function }~\forall~T_{h} \in \mathbb{T}_{h} \right\}.
\end{equation*}
\noindent Semi discrete formulation corresponding to the equation \eqref{16-9-22-1} is to seek $u_{h} \in X_{h}$ such that the following equations hold for $a.e. ~t\in (0,T]$ and  for all $v_{h} \in X_{h}$
\begin{equation}\label{semidiscrete}
 \begin{aligned}
     \left(^{C}D^{\alpha}_{t}u_{h},v_{h}\right)+M\left(\|\nabla u_{h}\|^{2}\right)(\nabla u_{h},\nabla v_{h})&=(f_{h},v_{h})+\int_{0}^{t}B(t,s,u_{h}(s),v_{h})~ds,\\
     u_{h}(x,0)&=u_{h}^{0}(x)~~\text{in}~~\mathbb{T}_{h},
 \end{aligned}
 \end{equation}
 where $f_{h}$ is the  $L^{2}$-projection of $f$ into $X_{h}$  and $u_{h}^{0}$  is  the Ritz projection of $u_{0}$ into $X_{h}$. Recall that, for any function $u$ the $L^{2}$-projection $P_{h}u$  is defined  by
  \begin{equation}
    (P_{h}u,v_{h})=(u,v_{h})~~\text{for all}~ v_{h}~\text{in}~X_{h},\label{A2a}
\end{equation}
which satisfies the following $H^{1}$ stability result \cite{bramble2002stability}
\begin{equation}\label{14-9-22-4-2}
\|\nabla P_{h}v\|\lesssim \|\nabla v\|~\forall~v~\in H^{1}_{0}(\Omega).
\end{equation}
\noindent and the Ritz projection $R_{h}u$ is defined by
 \begin{equation}
    (\nabla R_{h}u, \nabla v_{h})=(\nabla u,\nabla v_{h})~~\text{for all}~ v_{h}~\text{in}~X_{h}.~ \label{A2}
\end{equation}
\begin{thm} \label{a priori bounds-1 } \cite{lalit} Suppose that $(A1)$ and $(A2)$ hold. Then there exists a unique solution to the problem \eqref{semidiscrete}  which  satisfies the following a priori bounds
\begin{align}
	\|u_{h}\|^{2}_{L^{\infty}\left(0,T;L^{2}(\Omega)\right)}+\|u_{h}\|^{2}_{L^{2}_{\alpha}(0,T;H^{1}_{0}(\Omega))} &\lesssim \left(\| u_{0}\|^{2}+\|f\|^{2}_{L^{\infty}(0,T;L^{2}(\Omega))}\right).\label{AMP1as1-1}\\
\|u_{h}\|^{2}_{L^{\infty}\left(0,T;H^{1}_{0}(\Omega)\right)}& \lesssim \left(\| \nabla u_{0}\|^{2}+\|f\|^{2}_{L^{\infty}(0,T;L^{2}(\Omega))}\right).\label{AMP1as-1}
\end{align}
\end{thm}
 \noindent For the semi discrete error analysis of  the equation \eqref{16-9-22-1}, we modify the standard Ritz-Volterra projection operator \cite{cannon1988non} in such a way that it simplifies the complexities arising from the nonlocal nonlinear diffusion coefficient.
 The modified Ritz-Volterra projection operator $(M_{h}u):[0,T]\rightarrow X_{h}$  satisfies the following equation for all $v_{h}$ in $X_{h}$ and $a.e. ~t$ in $[0,T]$
\begin{equation}\label{2.5}
M\left(\|\nabla u\|^{2}\right)\left(\nabla\left (u-M_{h}u\right),\nabla v_{h}\right)=\int_{0}^{t}B\left(t,s,u(s)-M_{h}u(s),v_{h}\right)~ds,
\end{equation} 
where $u$ is the solution of  the equation \eqref{16-9-22-1}.
\begin{thm}\label{well defined of modified projection operator} The modified Ritz-Volterra projection operator \eqref{2.5} is well defined.
\end{thm}
\begin{proof}
The equation \eqref{2.5} is actually a system of integral equations of Volterra type. We are finding $M_{h}u$ in a finite dimensional subspace $X_{h}$, say dimension of $X_{h}$ is $\bar{N}$. Let $\{\psi_{i}\}_{i=1}^{\bar{N}}$ be a basis for $X_{h}$  then by the identification
\begin{equation*}
   (M_{h}u)(x,t)=\sum_{i=1}^{\bar{N}}\beta_{i}(t)\psi_{i}(x),
\end{equation*} 
using which the equation \eqref{2.5} may be  rewritten as 
\begin{equation}\label{S1.8}
A(t)\beta(t)-\int_{0}^{t}B(t,s)\beta(s)~ds=F(t)
\end{equation}
where $A(t)$ and $B(t,s)$ are the matrices and $\beta(t)$ and $F(t)$ are the vectors given  by 
\begin{equation*}
\begin{aligned}
\left[A(t)\right]_{ij}&=M\left(\|\nabla u\|^{2}\right)\left(\nabla \psi_{i},\nabla \psi_{j}\right)~~\text{for}~1\leq i,j \leq \bar{N},\\
\left[B(t,s)\right]_{ij}&=B(t,s,\psi_{i}(s),\psi_{j})~~\text{for}~1\leq i,j \leq \bar{N},\\
\beta(t)&=\left(\beta_{1}(t),\beta_{2}(t),\beta_{3}(t),\dots,\beta_{{\bar{N}}}(t)\right)',\\
F(t)&=\left(F_{1}(t),F_{2}(t),F_{3}(t),\dots,F_{\bar{N}}(t)\right)',
\end{aligned}
\end{equation*}
with 
\begin{equation*}
F_{j}(t)=M\left(\|\nabla u\|^{2}\right)\left(\nabla u,\nabla \psi_{j}\right)-\int_{0}^{t}B(t,\tau,u(\tau),\psi_{j})~d\tau~~\text{for}~1\leq j \leq \bar{N}.
\end{equation*}
The matrix $A(t)$ is positive definite by (A2) which shows that the system \eqref{S1.8} possesses a unique solution $\beta(t)$ \cite[Theorem 2.1.2]{brunner2004collocation}. Consequently, $M_{h}u$ in the equation \eqref{2.5} is well defined. Further $(M_{h}u)(0)=R_{h}u_{0}$, where $R_{h}u_{0}$ is the Ritz projection of $u_{0}$ into $X_{h}$ defined by \eqref{A2}.
\end{proof}
\begin{lem}\cite{kumar2020finite} The modified Ritz-Volterra projection operator defined by \eqref{2.5} satisfies the following stability results
\begin{equation}\label{16-10-22-7}
\|\nabla M_{h}u\|\lesssim K~\text{and}~\|\Delta_{h}M_{h}u\|\lesssim K, ~\text{for all}~t~\in ~[0,T], 
\end{equation}
where $\Delta_{h} : X_{h}\rightarrow X_{h}$ is the discrete Laplacian operator defined by  
	\begin{equation}\label{discrete laplace}
	    (-\Delta_{h}u_{h},v_{h}):=(\nabla u_{h},\nabla v_{h})~~\forall~~ u_{h},v_{h} \in X_{h}.
	\end{equation} 
\end{lem}
\noindent Now, we prove the best approximation properties of the modified Ritz-Volterra projection operator using the approximation properties of  the Ritz projection operator.  The following results on the approximation properties of the Ritz projection for any function $u$ in $H^{2}(\Omega)\cap H^{1}_{0}(\Omega)$ are well known  \cite{thomee2007galerkin} 
\begin{align}
    \|R_{h}u-u\|+h\|\nabla\left(R_{h}u-u\right)\|&\lesssim h^{2}\|u\|_{2}.\label{L1}
    \end{align}
    \begin{thm}\label{best approximation} Suppose that the solution $u$ of the equation \eqref{16-9-22-1}  satisfies regularity assumption \eqref{14-9-22-1}. Then 
the modified Ritz-Volterra projection operator $M_{h}u$ of $u$ has  the following best approximation property 
\begin{align}
    \|u-M_{h}u\|+h\|\nabla(u-M_{h}u)\|\lesssim~h^{2}.\label{A10}
    \end{align}
\end{thm}
\begin{proof}
Denote $\rho:= u-M_{h}u$. Using Ritz projection operator \eqref{A2}, we may decompose $\rho$ as 
\begin{equation}\label{6-5-22}
    \rho=\phi+\psi, ~\text{where}~\phi=u-R_{h}u~\text{and}~\psi=R_{h}u-M_{h}u.
\end{equation}
\textbf{(Estimate on $\|\nabla \rho\|$):}
Putting $\rho=\phi+\psi$ and $v_{h}=\psi(t)$ in the definition  of modified Ritz-Volterra projection operator \eqref{2.5},  we get 
\begin{equation*}\label{A2.51}
M\left(\|\nabla u\|^{2}\right)\left(\nabla \phi,\nabla \psi\right)+M\left(\|\nabla u\|^{2}\right)\left(\nabla \psi,\nabla \psi\right)=\int_{0}^{t}B\left(t,s,\phi(s)+\psi(s),\psi(t)\right)~ds.
\end{equation*} 
 Use the   definition  of Ritz projection \eqref{A2}, positivity of diffusion coefficient (A2), and  continuity of the function $B\left(t,s,\phi(s)+\psi(s),\psi(t)\right)$ \eqref{2.3} to obtain 
\begin{equation}\label{A2.6}
m_{0}\|\nabla \psi(t)\|^{2}\leq \|\nabla \psi(t)\|\int_{0}^{t}B_{0}\|\nabla \phi(s)+\nabla \psi(s)\|ds.
\end{equation} 
Cancellation of  $\|\nabla \psi(t)\|$  from  both sides of \eqref{A2.6} and triangle inequality leads to
\begin{equation*}\label{1A2.6}
\|\nabla \psi(t)\|\lesssim \int_{0}^{t}\|\nabla \phi(s)\|~ds+\int_{0}^{t}\|\nabla \psi(s)\|~ds.
\end{equation*} 
Apply the Gr\"{o}nwall's inequality and the approximation property of Ritz projection  \eqref{L1} to conclude
\begin{equation}\label{A2.61}
\|\nabla \psi(t)\|\lesssim \int_{0}^{t}h\|u(s)\|_{2}~ds~~~a.e.~t \in (0,T).
\end{equation} 
Thus, approximation property of Ritz projection  \eqref{L1} and estimate \eqref{A2.61} imply 
\begin{equation*}\label{L2a}
\begin{aligned}
    \|\nabla \rho(t)\| \leq \|\nabla \phi(t)\|+\|\nabla \psi(t)\|\lesssim h\left(\|u(t)\|_{2}+\int_{0}^{t}\|u(s)\|_{2}~ds\right)~~~\text{a.e.}~t \in (0,T).
    \end{aligned}
\end{equation*}
Finally, regularity estimate \eqref{14-9-22-1} yields
\begin{equation}\label{L2}
\begin{aligned}
    \|\nabla \rho(t)\| \lesssim h ~~~a.e.~t \in (0,T).
    \end{aligned}
\end{equation}
\textbf{(Estimate on $\|\rho\|$):} We estimate $\|\rho\|$ using duality argument. Let $g\in L^{2}(\Omega)$ and $\tilde{g} \in H^{2}(\Omega)$ be the solution of the following nonlocal  elliptic problem for each $t \in [0,T]$
\begin{equation}\label{A3}
    \begin{aligned}
    -M\left(\|\nabla u\|^{2}\right)\Delta \tilde{g}&=g \quad \text{in}~\Omega,\\
    \tilde{g}&=0\quad \text{on}~\partial \Omega.
    \end{aligned}
\end{equation}
The problem \eqref{A3} is well-posed by Lax-Milgram Lemma and by standard elliptic regularity result, we have 
\begin{equation}\label{L3}
    \|\tilde{g}\|_{2}\lesssim \|g\|.
\end{equation}
Using \eqref{A3} and integration by parts over $\Omega$ one obtain
\begin{equation}\label{A4a}
\begin{aligned}
    (\rho,g) =M\left(\|\nabla u\|^{2}\right)(\nabla \rho,\nabla \tilde{g}).
    \end{aligned}
\end{equation}
Adding and subtracting $M\left(\|\nabla u\|^{2}\right)(\nabla \rho,\nabla R_{h}\tilde{g} )$ in R.H.S. of \eqref{A4a} leads to
\begin{equation*}\label{A4b}
\begin{aligned}
    (\rho,g) =M\left(\|\nabla u\|^{2}\right)(\nabla \rho,\nabla \tilde{g}-\nabla R_{h}\tilde{g} )+M\left(\|\nabla u\|^{2}\right)(\nabla \rho,\nabla R_{h}\tilde{g}).
    \end{aligned}
\end{equation*}
The definition of modified Ritz-Volterra projection operator \eqref{2.5} yields
\begin{equation}\label{A4}
\begin{aligned}
    (\rho,g) =M\left(\|\nabla u\|^{2}\right)(\nabla \rho,\nabla \tilde{g}-\nabla R_{h}\tilde{g} )+\int_{0}^{t}B\left(t,s,\rho(s),R_{h}\tilde{g}\right)~ds.
    \end{aligned}
\end{equation}
\noindent The equation \eqref{A4} may be  rewritten as 
\begin{equation}\label{A43}
\begin{aligned}
    (\rho,g)&=M\left(\|\nabla u\|^{2}\right)(\nabla \rho,\nabla \tilde{g}-\nabla R_{h}\tilde{g} )+\int_{0}^{t}B\left(t,s,\rho(s),R_{h}\tilde{g}-\tilde{g}\right)ds\\
    &+\int_{0}^{t}B\left(t,s,\rho(s),\tilde{g}\right)ds.
    \end{aligned}
\end{equation}
Consider the  integrand from the last term of R.H.S. of \eqref{A43} 
\begin{equation}\label{A43-1}
    B\left(t,s,\rho(s),\tilde{g}\right)= (b_{2}(x,t,s)\nabla \rho(s),\nabla \tilde{g})-(b_{1}(x,t,s)\rho(s),\nabla \tilde{g})+(b_{0}(x,t,s)\rho(s),\tilde{g}).
\end{equation}
Integration by parts in the first term of R.H.S. of \eqref{A43-1} yields 
\begin{equation}\label{A43-2}
    B\left(t,s,\rho(s),\tilde{g}\right)= -( \rho(s),\nabla\cdot(b_{2}^{'}(x,t,s)\nabla \tilde{g}))-(b_{1}(x,t,s)\rho(s),\nabla \tilde{g})+(b_{0}(x,t,s)\rho(s),\tilde{g}).
\end{equation}
Since $b_{2}$ is a symmetric matrix so we have
\begin{equation*}\label{A43-3}
    B\left(t,s,\rho(s),\tilde{g}\right)= -( \rho(s),\nabla\cdot(b_{2}(x,t,s)\nabla \tilde{g}))-(b_{1}(x,t,s)\rho(s),\nabla \tilde{g})+(b_{0}(x,t,s)\rho(s),\tilde{g}).
\end{equation*}
This implies
\begin{equation}\label{16-10-22-1}
    |B\left(t,s,\rho(s),\tilde{g}\right)|\lesssim  \|\rho(s)\|~\|\Delta \tilde{g}\|+\|\rho(s)\|~\| \nabla \tilde{g}\|+\|\rho(s)\|~\|\tilde{g}\|\lesssim \|\rho(s)\|~\|\tilde{g}\|_{2}.
\end{equation}
Apply  Cauchy-Schwarz inequality in \eqref{A43} and estimate \eqref{16-10-22-1} to get 
\begin{equation*}\label{A6}
    |(\rho,g)|\lesssim \|\nabla \rho\|\|\nabla \tilde{g}-\nabla R_{h}\tilde{g}\| +\|\nabla R_{h}\tilde{g}-\nabla \tilde{g}\|\int_{0}^{t}\|\nabla \rho(s)\|ds\\
    +\|\tilde{g}\|_{2}\int_{0}^{t}\|\rho(s)\|~ds.
\end{equation*}
In the view of estimates \eqref{L1}, \eqref{L2}, and \eqref{L3}, we  conclude
 \begin{equation*}\label{A61}
 \begin{aligned}
    |(\rho,g)| &\lesssim \left(h^{2} +\int_{0}^{t}\|\rho(s)\|~ds\right)\|\tilde{g}\|_{2} \lesssim \left(h^{2} +\int_{0}^{t}\| \rho(s)\|~ds\right)\|g\|
    \end{aligned}
\end{equation*}
Recall the definition of $L^{2}$ norm 
\begin{equation*}
\begin{aligned}
    \|\rho\|:=\sup_{g \in L^{2}(\Omega)}\frac{|(\rho,g)|}{\|g\|}&\lesssim h^{2}+\int_{0}^{t}\|\rho(s)\|~ds.
     \end{aligned}
\end{equation*}
 Gr\"{o}nwall's inequality and regularity estimate \eqref{14-9-22-1} imply
\begin{equation}\label{A9}
   \|\rho\| \lesssim h^{2}.
\end{equation}
Finally, combining the estimates \eqref{L2} and \eqref{A9}, we deduce the desired estimate \eqref{A10}.
\end{proof}
\begin{thm}\label{mn} Let the solution $u$ of the equation \eqref{16-9-22-1} satisfies the regularity \eqref{14-9-22-1}-\eqref{14-9-22-2}. Assume that the  derivative of nonlocal diffusion coefficient $M$ is bounded. Then we have the following best approximation property
\begin{equation}
     \|~^{C}D_{t}^{\alpha}(u-M_{h}u)\|+h\|\nabla \left(~^{C}D^{\alpha}_{t}(u-M_{h}u)\right)\|\lesssim ~h^{2}\label{B10}.
\end{equation}
\end{thm}
\begin{proof}  Decompose $\rho$ \eqref{6-5-22} as in the proof of Theorem \ref{best approximation}.\\
\noindent \textbf{(Estimate on $\|\nabla \left(^{C}{D}_{t}^{\alpha}\rho\right)\|$):} First we estimate $\|\nabla \rho_{t}\|$ as 
\begin{equation}\label{imp}
    \| \nabla \rho_{t}\|^{2}=\left( \nabla \rho_{t}, \nabla \rho_{t}\right)=\left(\nabla \rho_{t}, \nabla \phi_{t}\right)+\left( \nabla \rho_{t},\nabla \psi_{t}\right) \leq \|\nabla \rho_{t}\|\|\nabla \phi_{t}\|+\left(\nabla \rho_{t},\nabla \psi_{t}\right).
\end{equation}
\noindent The approximation property of Ritz projection \eqref{L1} and the regularity assumptions \eqref{14-9-22-1}-\eqref{14-9-22-2} imply
\begin{equation}\label{first estimate }
\|\nabla \phi_{t}\| \leq h \|u_{t}\|_{2} \lesssim ht^{\alpha-1}.
\end{equation}
\noindent To estimate the $\left( \nabla \rho_{t},\nabla \psi_{t}\right)$, we  differentiate the modified Ritz-Volterra projection operator  \eqref{2.5} w.r.t $t$ as  
\begin{equation}\label{b1}
    \left(\frac{\partial}{\partial t}\left(M\left(\|\nabla u\|^{2}\right) \nabla \rho\right),\nabla v_{h}\right)= \frac{\partial}{\partial t}\int_{0}^{t}B(t,s,\rho(s),v_{h})~ds.
\end{equation}
Using Leibnitz rule of differentiation under integration in the R.H.S of the  equation \eqref{b1},  we get 
\begin{equation}\label{b21}
\begin{aligned}
     \left(M\left(\|\nabla u\|^{2}\right)\nabla \rho_{t},\nabla v_{h}\right)
    &=-\left(M'\left(\|\nabla u\|^{2}\right)2(\nabla u, \nabla u_{t})\nabla \rho,\nabla v_{h}\right)+B(t,t,\rho(t),v_{h})\\
    &+\int_{0}^{t}B_{t}(t,s,\rho(s),v_{h})~ds,
    \end{aligned}
\end{equation}
where
\begin{equation}
\begin{aligned}
    B_{t}(t,s,\rho(s),v_{h})&:= (b_{2t}(x,t,s)\nabla \rho(s),\nabla v_{h})+(\nabla\cdot(b_{1t}(x,t,s)\rho(s)),v_{h})\\
    &+(b_{0t}(x,t,s)\rho(s),v_{h}).
    \end{aligned}
\end{equation}
Apply (A2), continuity of memory operator \eqref{2.3}, and Cauchy-Schwarz inequality in \eqref{b21} to obtain
\begin{equation}\label{cq}
\begin{aligned}
    \left(\nabla \rho_{t},\nabla v_{h}\right)&\lesssim\|\nabla u\|\|\nabla u_{t}\|\|\nabla \rho\|\|\nabla v_{h}\|+\|\nabla \rho\|\|\nabla v_{h}\|+\|\nabla v_{h}\|\int_{0}^{t}\|\nabla \rho(s)\|~ds.
    \end{aligned}
\end{equation}
A priori bound on $\|\nabla u\|$ \eqref{AMP1as}, regularity estimate \eqref{14-9-22-2}  and estimate on $\|\nabla \rho\|$ \eqref{A10} yield 
\begin{equation}\label{p12}
    \left(\nabla \rho_{t},\nabla v_{h}\right)\lesssim h (t^{\alpha-1}+1)\|\nabla v_{h}\|\lesssim h t^{\alpha-1}\|\nabla v_{h}\|.
\end{equation}
As $\psi_{t} \in X_{h}$, thus we have
\begin{equation}\label{second estimate }
    \left( \nabla \rho_{t}, \nabla \psi_{t}\right) \lesssim ht^{\alpha-1}\|\nabla \psi_{t}\|.
\end{equation}
In the view of  estimates \eqref{first estimate } and \eqref{second estimate } in \eqref{imp}, we deduce 
\begin{equation*}
\begin{aligned}
    \|\nabla \rho_{t}\|^{2}&\lesssim ht^{\alpha-1}\| \nabla \rho_{t}\|+ht^{\alpha-1}\| \nabla \psi_{t}\|,\\
    &\lesssim ht^{\alpha-1} \|\nabla \rho_{t}\|+ht^{\alpha-1}\left( \|\nabla \rho_{t}\|+\| \nabla \phi_{t}\|\right).
    \end{aligned}
\end{equation*}
Apply estimate \eqref{first estimate } and Young's inequality to arrive at
\begin{equation}\label{time gradient approximation }
\|\nabla \rho_{t}\| \lesssim ht^{\alpha-1}.
\end{equation}
\noindent Further, consider 
\begin{equation*}
\begin{aligned}
    \|\nabla\left(~^{C}D^{\alpha}_{t}\rho\right)\|=\|\left(~^{C}D^{\alpha}_{t}\nabla \rho\right)\|&=\left\|\frac{1}{\Gamma(1-\alpha)}\int_{0}^{t}(t-s)^{-\alpha}\frac{\partial (\nabla \rho)}{\partial s}(s)~ds\right\|\\
    &\leq \frac{1}{\Gamma(1-\alpha)}\int_{0}^{t}(t-s)^{-\alpha}\|\nabla \frac{\partial  \rho}{\partial s}(s)\|ds.
    \end{aligned}
\end{equation*}
Using \eqref{time gradient approximation } and the definition of Beta function, we conclude
\begin{equation}\label{6-5-22-1}
\begin{aligned}
    \|\nabla\left(~^{C}D^{\alpha}_{t}\rho\right)\|& \leq \frac{h}{\Gamma(1-\alpha)}\int_{0}^{t}(t-s)^{-\alpha}s^{\alpha-1}~ds\\
    &\lesssim h.
    \end{aligned}
\end{equation}
\noindent \textbf{(Estimate on $\| \left(^{C}{D}_{t}^{\alpha}\rho\right)\|$):}
 We estimate $\|\rho_{t}\|$  using  duality argument as in the proof of estimate \eqref{A10}. Let $\tilde{g} \in H^{2}(\Omega)$ be the solution of the problem \eqref{A3} then 
\begin{equation}\label{L2 bound}
    (\rho_{t},g)=\left(M\left(\|\nabla u\|^{2}\right)\nabla \rho_{t},\nabla \tilde{g}-\nabla R_{h}\tilde{g}\right)+\left(M\left(\|\nabla u\|^{2}\right)\nabla \rho_{t},\nabla R_{h} \tilde{g}\right).
\end{equation}
Using \eqref{b21}, we have 
\begin{equation}\label{16-10-22-2}
    \begin{aligned}
        \left(M\left(\|\nabla u\|^{2}\right)\nabla \rho_{t},\nabla R_{h} \tilde{g}\right)&=-\left(M'\left(\|\nabla u\|^{2}\right)2(\nabla u, \nabla u_{t})\nabla \rho,\nabla R_{h} \tilde{g}-\nabla \tilde{g}\right)\\
        &-\left(M'\left(\|\nabla u\|^{2}\right)2(\nabla u, \nabla u_{t})\nabla \rho,\nabla \tilde{g}\right)+B(t,t,\rho(t),\tilde{g})\\
        &+B(t,t,\rho(t),R_{h} \tilde{g}-\tilde{g})+\int_{0}^{t}B_{t}(t,s,\rho(s),\tilde{g})~ds\\
        &+\int_{0}^{t}B_{t}(t,s,\rho(s),R_{h}\tilde{g}-\tilde{g})~ds.
    \end{aligned}
\end{equation}
Substitute \eqref{16-10-22-2} in \eqref{L2 bound} to obtain 
\begin{equation}\label{L2 bound-1}
\begin{aligned}
    (\rho_{t},g)&=\left(M\left(\|\nabla u\|^{2}\right)\nabla \rho_{t},\nabla \tilde{g}-\nabla R_{h}\tilde{g}\right)\\
    &-\left(M'\left(\|\nabla u\|^{2}\right)2(\nabla u, \nabla u_{t})\nabla \rho,\nabla R_{h} \tilde{g}-\nabla \tilde{g}\right)\\
        &-\left(M'\left(\|\nabla u\|^{2}\right)2(\nabla u, \nabla u_{t})\nabla \rho,\nabla \tilde{g}\right)+B(t,t,\rho(t),\tilde{g})\\
        &+B(t,t,\rho(t),R_{h} \tilde{g}-\tilde{g})+\int_{0}^{t}B_{t}(t,s,\rho(s),\tilde{g})~ds+\int_{0}^{t}B_{t}(t,s,\rho(s),R_{h}\tilde{g}-\tilde{g})~ds.
    \end{aligned}
\end{equation}
Cauchy-Schwarz inequality, a priori bound on $\|\nabla u\|$ \eqref{AMP1as}, and estimate \eqref{16-10-22-1} yield
\begin{equation}\label{L2 bound-2}
\begin{aligned}
    |(\rho_{t},g)|&\lesssim \|\nabla \rho_{t}\|~\|\nabla \tilde{g}-\nabla R_{h}\tilde{g}\|+\|\nabla u_{t}\|~\|\nabla \rho\|~\|\nabla R_{h} \tilde{g}-\nabla \tilde{g}\|+\| \nabla u_{t}\|~\| \rho\|~\|\Delta \tilde{g}\|\\
    &+\|\rho\|~\|\tilde{g}\|_{2}+\|\nabla \rho\|~\|\nabla(R_{h} \tilde{g}-\tilde{g})\|+\|\tilde{g}\|_{2}\int_{0}^{t}\|\rho(s)\|~ds\\
    &+\|\nabla R_{h}\tilde{g}-\nabla \tilde{g}\|\int_{0}^{t}\|\nabla \rho(s)\|~ds.
    \end{aligned}
\end{equation}
Apply estimates \eqref{time gradient approximation }, \eqref{L1}, \eqref{L3}, \eqref{14-9-22-2}, \eqref{A10} to reach at 
\begin{equation}
\begin{aligned}
    |(\rho_{t},g)|&\lesssim h^{2}t^{\alpha-1}\|g\|+h^{2}t^{\alpha-1}\|g\|+ h^{2}t^{\alpha-1}\|g\|+h^{2}\|g\|+h^{2}\|g\|+\|g\|h^{2}+h^{2}\|g\|.
    \end{aligned}
\end{equation}
This implies 
\begin{equation}\label{cap}
    \|\rho_{t}\| := \sup_{g \in L^{2}(\Omega)}\frac{|(\rho_{t},g)|}{\|g\|}\lesssim h^{2}t^{\alpha-1}.
\end{equation}
 \noindent Finally, consider 
\begin{equation*}
\begin{aligned}
    \|~^{C}D^{\alpha}_{t}\rho\|=\left\|\frac{1}{\Gamma(1-\alpha)}\int_{0}^{t}(t-s)^{-\alpha}\frac{\partial \rho}{\partial s}ds\right\|&\leq \frac{1}{\Gamma(1-\alpha)}\int_{0}^{t}(t-s)^{-\alpha}\|\frac{\partial \rho}{\partial s}(s)\|ds.
    \end{aligned}
\end{equation*}
Using \eqref{cap} and the definition of Beta function, we get 
\begin{equation}\label{16-10-22-3}
    \|~^{C}D^{\alpha}_{t}\rho\|\lesssim h^{2} \frac{1}{\Gamma(1-\alpha)}\int_{0}^{t}(t-s)^{-\alpha} s^{\alpha-1}ds\lesssim h^{2}. 
    \end{equation}
We combine the estimates \eqref{6-5-22-1} and \eqref{16-10-22-3} to conclude the result \eqref{B10}.
\end{proof}
\begin{thm} Suppose that the solution $u$ of the equation \eqref{16-9-22-1} and the solution $u_{h}$ of the equation \eqref{semidiscrete} satisfy the regularity estimate \eqref{14-9-22-1} and \eqref{14-9-22-2}. Then we have the following semi discrete error estimates 
\begin{equation}\label{cap2}
    \|u-u_{h}\|_{L^{\infty}\left(0,T;L^{2}(\Omega)\right)}+\|u-u_{h}\|_{L^{2}_{\alpha}(0,T;H^{1}_{0}(\Omega))} \lesssim h,
\end{equation}
and
\begin{equation}\label{cap3}
    \|u-u_{h}\|_{L^{\infty}\left(0,T;H^{1}_{0}(\Omega)\right)} \lesssim h.
\end{equation}
\end{thm}
\begin{proof} We write $u-u_{h}=\rho+\theta$ where 
\begin{equation*}
    \rho=u-M_{h}u \quad \text{and}\quad \theta =M_{h}u-u_{h}.
\end{equation*}
 Using the  semi discrete formulation \eqref{semidiscrete} with $u_{h}=M_{h}u-\theta$ one  have
	\begin{equation}\label{6-5-22-2}
	\begin{aligned}
	\left(^{C}D^{\alpha}_{t}(M_{h}u-\theta),v_{h}\right) &+M\left(\|\nabla u_{h}\|^{2}\right)(\nabla (M_{h}u-\theta),\nabla v_{h})\\
	&=(f_{h},v_{h})+\int_{0}^{t}B(t,s,M_{h}u(s)-\theta(s),v_{h})ds.
	\end{aligned}
	\end{equation}
	Since $f$ belongs to $L^{2}(\Omega)$ then we have  $(f,v_{h})=(f_{h},v_{h})$ for all $v_{h}$ in $X_{h}$. Using weak formulation \eqref{nonlinear weak problem} and definition of modified Ritz-Volterra projection \eqref{2.5} in equation \eqref{6-5-22-2}, we get
		\begin{equation}\label{6-5-22-2-1}
	\begin{aligned}
	\left(^{C}D^{\alpha}_{t}\theta,v_{h}\right)& +M\left(\|\nabla u_{h}\|^{2}\right)(\nabla \theta,\nabla v_{h})\\
	&=-\left(^{C}D^{\alpha}_{t}\rho,v_{h}\right)+\left[M\left(\|\nabla u_{h}\|^{2}\right)-M\left(\|\nabla u\|^{2}\right)\right](\nabla M_{h}u,\nabla v_{h})\\
	&+\int_{0}^{t}B(t,s,\theta(s),v_{h})~ds.
	\end{aligned}
	\end{equation}
Set $v_{h}=\theta$ in \eqref{6-5-22-2-1} and apply positivity along with Lipschitz continuity of diffusion coefficient (A2) to obtain 
	\begin{equation}\label{6-5-22-2-1-1}
	\begin{aligned}
	\left(^{C}D^{\alpha}_{t}\theta,\theta\right)& +m_{0}\|\nabla \theta\|^{2}\\
	&\lesssim \|~^{C}D^{\alpha}_{t}\rho\|\|\theta\|+L_{M}(\|\nabla u_{h}\|+\|\nabla u\|)~(\|\nabla \rho\|+\|\nabla \theta\|)~\|\nabla M_{h}u\|~\|\nabla \theta\|\\
	&+\|\nabla \theta \|\int_{0}^{t}\|\nabla \theta(s)\|~ds.
	\end{aligned}
	\end{equation}
We use a priori bounds on $\|\nabla u_{h}\|$ \eqref{AMP1as-1}, $\|\nabla u\| $ \eqref{AMP1as}, $\|\nabla M_{h}u\|$ \eqref{16-10-22-7} along with Young's inequality to arrive at 
	\begin{equation}\label{6-5-22-2-1-2}
	\begin{aligned}
	\left(^{C}D^{\alpha}_{t}\theta,\theta\right)+\left(m_{0}-4L_{M}K^{2}\right)\|\nabla \theta\|^{2} &\lesssim \|~^{C}D^{\alpha}_{t}\rho\|^{2}+\|\theta\|^{2}+\|\nabla \rho\|^{2}+ \int_{0}^{t}\|\nabla \theta(s)\|^{2}~ds.
	\end{aligned}
	\end{equation}
	 Apply the same arguments as we prove estimate \eqref{AMP1as1} to deduce
	\begin{equation}\label{16-10-22-6}
\|\theta\|^{2}_{L^{\infty}\left(0,T;L^{2}(\Omega)\right)}+\left(l\ast\|\nabla \theta\|^{2}\right)(t) \lesssim \left[l\ast \left(\|\nabla \rho\|^{2}+\|~^{C}D^{\alpha}_{t}\rho\|^{2}\right)\right](t)+\|\nabla \theta(0)\|^{2},
	\end{equation}
where $l(t)=\frac{t^{\alpha-1}}{\Gamma(\alpha)}$. Since $u_{h}^{0}=(M_{h}u)(0)$, thus $\theta(0)=0$. We apply the best approximation properties of modified Ritz-Volterra projection operator as in Theorem \ref{best approximation} and Theorem \ref{mn} along  with triangle inequality to conclude 
	\begin{equation*}
	\|u-u_{h}\|_{L^{\infty}\left(0,T;L^{2}(\Omega)\right)}+\|u-u_{h}\|_{L^{2}_{\alpha}\left(0,T;H^{1}_{0}(\Omega)\right)} \lesssim h.
	\end{equation*}
To prove the estimate \eqref{cap3}, we define another  discrete Laplacian operator $\Delta_{h}^{b_{2}} : X_{h}\rightarrow X_{h}$ which satisfies
\begin{equation}\label{discrete laplace-1}
	    (-\Delta_{h}^{b_{2}}~ u_{h},v_{h}):=(b_{2}(x,t,s)\nabla u_{h},\nabla v_{h})~~\forall~~ u_{h},v_{h} \in X_{h}, ~\forall~t,s~\in~[0,T].
	\end{equation}
	This operator is well defined as $b_{2}$ is a real symmetric and positive definite matrix. 
	Using the discrete Laplacian operators \eqref{discrete laplace} and \eqref{discrete laplace-1} we can rewrite the equation \eqref{6-5-22-2-1} as 
	\begin{equation}\label{6-5-22-2-3}
	\begin{aligned}
	&\left(^{C}D^{\alpha}_{t}\theta,v_{h}\right)+M\left(\|\nabla u_{h}\|^{2}\right)(-\Delta_{h}\theta, v_{h})\\
	&=-\left(P_{h}~^{C}D^{\alpha}_{t}\rho,v_{h}\right)+\left[M\left(\|\nabla u_{h}\|^{2}\right)-M\left(\|\nabla u\|^{2}\right)\right](-\Delta_{h} M_{h}u, v_{h})\\
	&+\int_{0}^{t}(-\Delta_{h}^{b_{2}}~\theta(s),v_{h})~ds+\int_{0}^{t}(\nabla \cdot(b_{1}(x,t,s)\theta(s)),v_{h})~ds+\int_{0}^{t}(b_{0}(x,t,s)\theta(s),v_{h})~ds.
	\end{aligned}
	\end{equation}
	Put $v_{h}=-\Delta_{h}\theta$ in \eqref{6-5-22-2-3} and apply positivity and Lipschitz continuity of diffusion coefficient (A2) together with Cauchy-Schwarz and Young's inequality to obtain
	\begin{equation}\label{6-5-22-2-4}
	\begin{aligned}
	&\left(^{C}D^{\alpha}_{t}\nabla \theta,\nabla \theta \right)+\|\Delta_{h}\theta\|^{2}\\
	&\lesssim \|\nabla P_{h}~^{C}D^{\alpha}_{t}\rho\|^{2}+\|\nabla \theta\|^{2}+L_{M}^{2}(\|\nabla u_{h}\|+\|\nabla u\|)^{2}~(\|\nabla \rho\|+\|\nabla \theta \|)^{2}~\|\Delta_{h} M_{h}u\|^{2}\\
	&+\int_{0}^{t}\|\Delta_{h}^{b_{2}}~\theta(s)\|^{2}~ds+\int_{0}^{t}\|\nabla \theta(s)\|^{2}~ds+\int_{0}^{t}\|\theta(s)\|^{2}~ds.
	\end{aligned}
	\end{equation}
	Now, we find the estimate on $\|\Delta_{h}^{b_{2}}~\theta(s)\|$ as 
	\begin{equation}
	\begin{aligned}
	    |(-\Delta_{h}^{b_{2}}~\theta(s),v_{h})|&=|(b_{2}(x,t,s)\nabla \theta(s),\nabla v_{h})|\\
	    &\lesssim \|\nabla \theta(s)\|~\|\nabla v_{h}\|.
	    \end{aligned}
	\end{equation}
	This implies 
	\begin{equation}\label{16-10-22-5}
	    \|\Delta_{h}^{b_{2}}~\theta(s)\|:=\sup_{v_{h}\in X_{h} \subset H^{1}_{0}(\Omega)}\frac{|(-\Delta_{h}^{b_{2}}~\theta(s),v_{h})|}{\|v_{h}\|_{1}}\lesssim \|\nabla \theta(s)\|.
	\end{equation}
	Use the estimates on $\|\nabla u_{h}\|$\eqref{AMP1as-1}, $\|\nabla u\|$\eqref{AMP1as}, $\|\Delta_{h}M_{h}u\|$ \eqref{16-10-22-7},$\|\Delta_{h}^{b_{2}}~\theta(s)\|$ \eqref{16-10-22-5} and \eqref{14-9-22-4-2} in \eqref{6-5-22-2-4} to reach at
		\begin{equation}\label{6-5-22-2-5}
	\begin{aligned}
	\left(^{C}D^{\alpha}_{t}\nabla \theta,\nabla \theta \right)+\|\Delta_{h}\theta\|^{2}
	&\lesssim \|\nabla~^{C}D^{\alpha}_{t}\rho\|^{2}+\|\nabla \rho\|^{2}+\|\nabla \theta \|^{2}+\int_{0}^{t}\|\nabla \theta(s)\|^{2}~ds\\
	&+\int_{0}^{t}\|\theta(s)\|^{2}~ds.
	\end{aligned}
	\end{equation}
	Employ approximation properties \eqref{B10}, \eqref{A10} and estimate \eqref{16-10-22-6} to deduce
		\begin{equation}\label{6-5-22-2-5-1}
	\begin{aligned}
	\left(^{C}D^{\alpha}_{t}\nabla \theta,\nabla \theta \right)+\|\Delta_{h}\theta\|^{2}
	&\lesssim h^{2}+h^2+\|\nabla \theta \|^{2}.
	\end{aligned}
	\end{equation}
	Finally, follow the similar steps as we prove estimate  \eqref{cap2} to conclude \eqref{cap3}.
\end{proof}

\section{Linearized L1 Galerkin FEM on graded mesh}\label{15-10-22-3}
Let $N$ be a positive integer and set the nodal  points $\left( t_{0}<t_{1}<t_{2}<\dots< t_{N}\right)$  on the interval $[0,T]$ as
\begin{equation}\label{SU4}
    t_{n}=T\left(\frac{n}{N}\right)^{\delta}~~\text{for}~~n=0,1,2,\dots,N,
\end{equation}
where $\delta\geq 1$ is called grading parameter. The time step $\tau_{n}$ is given by  $\tau_{n}=t_{n}-t_{n-1}$ for $n=1,2,3,\dots,N$. The nodal points $t_{n}$ satisfy the following relation 
\begin{equation}\label{29-11-22-2}
    t_{n}\leq 2^{\delta} t_{n-1},~\text{for}~n\geq 2.
\end{equation}
\begin{prop}
The time step size $\tau_{n}$ for $n=1,2,3,\dots,N$ satisfy  the following estimate
\begin{equation}\label{am1}
    \tau_{n}\leq \delta TN^{-\delta}n^{\delta-1}\leq \delta T^{1/\delta}N^{-1}t_{n}^{(1-1/\delta)}.
\end{equation}
\end{prop}
\begin{proof}
 The estimate \eqref{am1} follows by applying mean value theorem to the function $ {\Phi}(x)=T\left(\frac{x}{N}\right)^{\delta}$ on $(n-1,n)$.
\end{proof}
\noindent The time steps $\tau_{n}$ satisfy the $M$-Conv property \cite{ren2021sharp}, i.e., there exists a positive constant $C_{\delta}$  such that for every $N$, we have
\begin{equation}\label{M conv prop }
    \frac{\tau_{n}}{\tau_{n-1}}\leq C_{\delta} \quad \text{for}\quad 2\leq n \leq N.
\end{equation}
\noindent  Let $u^{n}~(1\leq n \leq N)$ be the value of $u$ at $t_{n}$ then the Caputo fractional derivative of $u$ at $t_{n}$ is given by 
\begin{equation*}\label{LSA1}
\begin{aligned}
^{C}D_{t}^{\alpha}u^{n}&=\frac{1}{\Gamma(1-\alpha)} \int_{0}^{t_{n}}\left(t_{n}-s\right)^{-\alpha}\frac{\partial u}{\partial s}(s)~ds
=\frac{1}{\Gamma(1-\alpha)} \sum_{j=1}^{n}\int_{t_{j-1}}^{t_{j}}\left(t_{n}-s\right)^{-\alpha}\frac{\partial u}{\partial s}(s)~ds.
\end{aligned}
\end{equation*}
\noindent  The Caputo fractional derivative is approximated  by the L1 scheme \cite{kopteva2019error} as 
\begin{equation}\label{p3a}
\begin{aligned}
    ^{C}D^{\alpha}_{t}u^{n}\approx \frac{1}{\Gamma(1-\alpha)} \sum_{j=1}^{n}\left(\frac{u^{j}-u^{j-1}}{\tau_{j}}\right)\int_{t_{j-1}}^{t_{j}}\left(t_{n}-s\right)^{-\alpha}ds.
    \end{aligned}
\end{equation}
Thus, $^{C}D^{\alpha}_{t}u^{n}$ can be written as 
\begin{equation}\label{p3}
\begin{aligned}
    ^{C}D^{\alpha}_{t}u^{n}=  \mathbb{D}^{\alpha}_{t}u^{n}+\mathbb{T}^{n},
    \end{aligned}
\end{equation}
where $\mathbb{D}^{\alpha}_{t}u^{n}$ denotes the approximation \eqref{p3a} of Caputo fractional derivative  and $\mathbb{T}^{n}$ denotes the truncation error that  arises due to this approximation. We write $\mathbb{D}^{\alpha}_{t}u^{n}$  in a simpler way as 
\begin{equation}\label{approximation of Caputo Derivative}
    \mathbb{D}^{\alpha}_{t}u^{n}:=\sum_{j=1}^{n}k_{n,j}(u^{j}-u^{j-1})=k_{n,n}u^{n}-\sum_{j=1}^{n}(k_{n,j}-k_{n,j-1})u^{j-1},
\end{equation}
where $k_{n,j}$ are called discrete kernels, that are given by 
\begin{equation}\label{kernel}
    k_{n,j}:=\frac{\tau_{j}^{-1}}{\Gamma(1-\alpha)}\int_{t_{j-1}}^{t_{j}}(t_{n}-s)^{-\alpha}ds~~\text{for}~~1\leq  j \leq n, ~~~\text{and}~~k_{n,0}:=0.
\end{equation}
\noindent Based on the approximation \eqref{approximation of Caputo Derivative}, we develop the following linearized fully discrete numerical scheme for the  equation \eqref{16-9-22-1}.\\
\noindent Find $u_{h}^{n} \in X_{h}$  for $n=1,2,3,\dots,N$ which satisfy  the following equations  for all $v_{h}\in X_{h}$,
\begin{align}
    \left(\mathbb{D}^{\alpha}_{t}u_{h}^{n},v_{h}\right)+M\left(\|\nabla \tilde{u}_{h}^{n-1}\|^{2}\right)\left(\nabla u_{h}^{n},\nabla v_{h}\right)&=\left(f_{h}^{n},v_{h}\right)+\sum_{j=1}^{n-1}\tau_{j+1}\bar{B}\left(t_{n},t_{j},u_{h}^{j},v_{h}\right),~n\geq 2,\label{SU1}\\
     \left(\mathbb{D}^{\alpha}_{t}u_{h}^{1},v_{h}\right)+M\left(\|\nabla u_{h}^{1}\|^{2}\right)\left(\nabla u_{h}^{1},\nabla v_{h}\right)&=\left(f_{h}^{1},v_{h}\right)+\tau_{1}B\left(t_{1},t_{1},u_{h}^{1},v_{h}\right)~\text{for}~n=1,\label{SUM1}\\
     u_{h}^{0}&=R_{h}u_{0}~\text{in}~X_{h},\label{SUM2}\\
     u_{h}^{n}&=0~\text{on}~\partial \Omega,~  \text{for}~n=1,2,3,\dots,N, \label{SUM3}
    \end{align}
where $\tilde{u}_{h}^{n-1}$ and $\bar{B}\left(t_{n},t_{j},u_{h}^{j},v_{h}\right)$ used in \eqref{SU1} are defined as
\begin{equation}\label{linearization}
\tilde{u}_{h}^{n-1}=
\left(1+\frac{\tau_{n}}{\tau_{n-1}}\right)u_{h}^{n-1}-\frac{\tau_{n}}{\tau_{n-1}}u_{h}^{n-2},
\end{equation}
and 
\begin{equation}\label{S9}
\bar{B}\left(t_{n},t_{j},u_{h}^{j},v_{h}\right)=\begin{cases}
			\frac{1}{2}\left[B\left(t_{n},t_{j},u_{h}^{j},v_{h}\right)+B\left(t_{n},t_{j+1},u_{h}^{j+1},v_{h}\right)\right], & ~\text{for}~ 1\leq j \leq n-2\\
            B\left(t_{n},t_{n-1},u_{h}^{n-1},v_{h}\right), & ~\text{for}~j=n-1.
		 \end{cases}
\end{equation}
\noindent As  discussed in the introduction that the numerical schemes based on the Newton method demand high computational cost as well as huge computer storage to solve the nonlinear system at each time steps in $[0,T]$. Also, the Jacobian of the corresponding nonlinear system is not sparse because of the nonlocality of the diffusion coefficient \cite{gudi2012finite}. To reduce these costs, we propose a new linearization technique \eqref{linearization} on nonuniform time mesh for the nonlocal nonlinear diffusion coefficient. 
\par For the approximation of memory  term, we apply the right rectangle rule over $\left[t_{0},t_{1}\right]$, the motivation for choosing the right rectangle rule over $\left[t_{0},t_{1}\right]$ comes from  the weak singularity \eqref{14-9-22-2} of the solution $u$ at $t_{0}$. Next, for $n\geq 2$  a modified trapezoidal rule is applied, in which we apply the  composite trapezoidal rule on $[t_{1},t_{n-1}]$ and then the left rectangle rule over $[t_{n-1},t_{n}]$ to remain in the linearized setup.
\subsection{A priori bounds on  numerical solutions}
In this subsection, we derive a priori bounds on  numerical solutions of the fully discrete formulation  \eqref{SU1}-\eqref{SUM3}. These a priori bounds play an important role in proving well-posedness  and convergence analysis of the proposed numerical scheme.
\par The following propositions  help us to derive a priori bounds on  numerical solutions of the developed numerical scheme \eqref{SU1}-\eqref{SUM3}.
\begin{prop}\label{5-11-22-1} The approximation $\mathbb{D}^{\alpha}_{t}u^{n}$ of Caputo fractional derivative satisfies the following inequality 
\begin{equation}\label{derivative positivity}
    \left(\mathbb{D}^{\alpha}_{t}u^{n},u^{n}\right) \geq \frac{1}{2}\mathbb{D}^{\alpha}_{t}\|u^{n}\|^{2}.
\end{equation}
\end{prop}
\begin{proof}
Use the property of discrete kernels ($0< k_{n,j-1} < k_{n,j}$ for all admissible $j$ and $n$), and Young's inequality in the definition of $ \left(\mathbb{D}^{\alpha}_{t}u^{n},u^{n}\right)$ to prove the estimate  \eqref{derivative positivity}.
\end{proof}
\begin{prop}\label{5-11-22-2} For any  $n \in \{1,2,3,\dots, N\}$, the discrete kernels  $k_{n,1}$ defined in  \eqref{kernel} satisfy the following bound
\begin{equation}\label{tau1kn1 estimate }
    \tau_{n}^{\alpha}k_{n,1}\leq \frac{1}{\Gamma(1-\alpha)}.
\end{equation}
\end{prop}

\begin{proof} Let $\psi:(t_{0},t_{1})\rightarrow \mathbb{R}$ be the function defined by 
\begin{equation*}
    \psi(x)=\frac{-(t_{n}-x)^{1-\alpha}}{\Gamma (2-\alpha)},
\end{equation*}
then apply mean value theorem on the function $\psi$ to get the result \eqref{tau1kn1 estimate }.
\end{proof}

\begin{prop}\label{7-8-1}
For any $n \in \{2,3,4,\dots,N\}$ and $j=1,2,3,\dots,n-1$, define $w_{n,j}$ by
\begin{equation}
   w_{n,j}= \tau_{n}^{\alpha/2}\tau_{j+1}\tau_{j}^{-\alpha/2}+\tau_{n}^{\alpha}(k_{n,j+1}-k_{n,j}).
\end{equation}
Then $w_{n,j}>0$, and 
\begin{equation}\label{weight estimate}
    \sum_{j=1}^{n-1}w_{n,j}\lesssim 1+T.
\end{equation}
\end{prop}
\begin{proof} For $j \in \{1,2,3,\dots,n\}$, $\tau_{j}>0$ and the property of discrete kernel  $k_{n,j}> k_{n,j-1}$ implies $w_{n,j}>0.$ Further,
\begin{equation*}
    \begin{aligned}
        \sum_{j=1}^{n-1}w_{n,j}= \sum_{j=1}^{n-1}\left(\tau_{n}^{\alpha/2}\tau_{j+1}\tau_{j}^{-\alpha/2}+\tau_{n}^{\alpha}(k_{n,j+1}-k_{n,j})\right).
         \end{aligned}
\end{equation*}
Using M-Conv property \eqref{M conv prop }, one have 
\begin{equation*}
    \begin{aligned}
        \sum_{j=1}^{n-1}w_{n,j}\leq \sum_{j=1}^{n-1}\tau_{n}^{\alpha/2}C_{\delta}\tau_{j}^{1-\alpha/2}+\tau_{n}^{\alpha}(k_{n,n}-k_{n,1}).
        \end{aligned}
\end{equation*}
 Definition of $k_{n,n}$ and estimate \eqref{tau1kn1 estimate } yield
\begin{equation*}
    \begin{aligned}
        \sum_{j=1}^{n-1}w_{n,j}\lesssim  \tau_{n}^{\alpha/2} \sum_{j=1}^{n-1}\tau_{j}^{1-\alpha/2}+1.
    \end{aligned}
\end{equation*}
employing the property \eqref{am1} of the graded mesh, we obtain  
\begin{equation*}
    \begin{aligned}
        \sum_{j=1}^{n-1}w_{n,j}&\lesssim  N^{-\alpha/2}t_{n}^{(1-1/\delta)(\alpha/2)} \sum_{j=1}^{n-1}N^{-(1-\alpha/2)}t_{j}^{(1-1/\delta)(1-\alpha/2)}+1\\
      &\lesssim  N^{-\alpha/2}\sum_{j=1}^{n-1}N^{-(1-\alpha/2)}+1  \lesssim  1+N^{-1}NT=1+T.
    \end{aligned}
\end{equation*}
\end{proof}
\noindent To derive a priori bounds on numerical solutions of the fully discrete formulation \eqref{SU1}-\eqref{SUM3}, we also define a new weighted $H^{1}(\Omega)$ norm by 
\begin{equation}\label{new norm}
    \||u_{h}^{n}\||:=\|u_{h}^{n}\|+\tau^{\alpha/2}_{n}\|\nabla u_{h}^{n}\|~~~\text{for}~~~1\leq n \leq N,~~\text{and}~~\||u_{h}^{0}\||:=\|u_{h}^{0}\|.
\end{equation}
\begin{thm} Suppose that $(A1)$-$(A2)$ hold. Then the numerical solutions $u_{h}^{n},~(1\leq n \leq N)$ of the numerical scheme \eqref{SU1}-\eqref{SUM3} satisfy the following a priori bound
\begin{equation}\label{numerical bound}
    \max_{1\leq n \leq N}\||u_{h}^{n}\|| \lesssim \|\nabla u_{0}\|+\max_{1\leq n \leq N}\|f^{n}\|.
\end{equation}
\end{thm}
\begin{proof} For $n=1$,
put the value of  $\mathbb{D}^{\alpha}_{t}u_{h}^{1}$ in \eqref{SUM1} and denote $\gamma_{0}=\Gamma(2-\alpha)$ to get 
\begin{equation}\label{as}
    (u_{h}^{1}-u_{h}^{0},v_{h})+\tau_{1}^{\alpha}\gamma_{0}M\left(\|\nabla u_{h}^{1}\|^{2}\right)\left(\nabla u_{h}^{1},\nabla v_{h}\right)=\tau_{1}^{\alpha}\gamma_{0}\left(f_{h}^{1},v_{h}\right)+\tau_{1}^{\alpha}\gamma_{0}\tau_{1}B\left(t_{1},t_{1},u_{h}^{1},v_{h}\right).
\end{equation}
Set $v_{h}=u_{h}^{1}$ in \eqref{as}  and use positivity of diffusion coefficient (A2), continuity of memory operator \eqref{2.3} along with  Cauchy-Schwarz inequality to obtain
\begin{equation}\label{ADY1}
 \|u_{h}^{1}\|^{2}+\tau_{1}^{\alpha}\gamma_{0}m_{0}\|\nabla u_{h}^{1}\|^{2}\leq \|u_{h}^{0}\|\|u_{h}^{1}\|+\tau_{1}^{\alpha}\gamma_{0}\|f_{h}^{1}\|\| u_{h}^{1}\|+\tau_{1}^{1+\alpha}\gamma_{0}B_{0}\|\nabla u_{h}^{1}\|^{2}.
 \end{equation}
 Simplification of \eqref{ADY1} leads to  
 \begin{equation*}\label{ADY2}
\begin{aligned}
 \|u_{h}^{1}\|^{2}+\tau_{1}^{\alpha}\gamma_{0}(m_{0}-B_{0}\tau_{1})\|\nabla u_{h}^{1}\|^{2}&\leq \|u_{h}^{0}\|\|u_{h}^{1}\|+\tau_{1}^{\alpha}\gamma_{0}\|f_{h}^{1}\|\| u_{h}^{1}\|.
 \end{aligned}
 \end{equation*}
 For sufficiently small $\tau_{1}<\frac{m_{0}}{B_{0}}$ ~$\left(\text{say}~\tau_{1}=\frac{m_{0}}{2B_{0}} \right)$, we have 
 \begin{equation*}
     \|u_{h}^{1}\|^{2}+\tau_{1}^{\alpha}\|\nabla u_{h}^{1}\|^{2}\lesssim \|u_{h}^{0}\|\|u_{h}^{1}\|+\tau_{1}^{\alpha}\|f_{h}^{1}\|\|u_{h}^{1}\|\lesssim \|u_{h}^{0}\|\|u_{h}^{1}\|+\|f_{h}^{1}\|\|u_{h}^{1}\|.
 \end{equation*}
 Definition \eqref{new norm} of weighted $H^{1}(\Omega)$ norm yields
 \begin{equation}\label{ADY2-1}
     \||u_{h}^{1}\||^{2}\lesssim \left(\|u_{h}^{0}\|+\|f_{h}^{1}\|\right)\||u_{h}^{1}\||.
 \end{equation}
 Cancel $\||u_{h}^{1}\||$ from both sides of \eqref{ADY2-1} and apply Poincar\'{e} inequality to arrive at  
 \begin{equation}\label{18-10-22-1-12}
     \||u_{h}^{1}\||\lesssim \left(\|u_{h}^{0}\|+\|f_{h}^{1}\|\right)\lesssim \|\nabla u_{h}^{0}\|+\max_{1\leq n \leq N}\|f_{h}^{n}\|.
 \end{equation}
 Further, using $\|\nabla u_{h}^{0}\|\lesssim \|\nabla u_{0}\|$ and $\|f_{h}\|\lesssim \|f\|$ in \eqref{18-10-22-1-12}, we deduce
  \begin{equation}\label{18-10-22-1}
     \||u_{h}^{1}\||\lesssim \|\nabla u_{0}\|+\max_{1\leq n \leq N}\|f^{n}\|.
 \end{equation}
  If 
 \begin{equation}\label{18-10-22-2}
     \max_{1\leq n \leq N}\||u_{h}^{n}\||=\||u_{h}^{1}\||,
 \end{equation}
 then proof of the result  \eqref{numerical bound} follows from the estimate  \eqref{18-10-22-2}. Otherwise, there exists $m~\in~\{2,3,4,\dots,N\}$ such that
  \begin{equation}\label{18-10-22-3}
     \max_{1\leq n \leq N}\||u_{h}^{n}\||=\||u_{h}^{m}\||.
 \end{equation}
 \noindent Now, consider the case  $n\geq 2$  and substitute $v_{h}=u_{h}^{n}$ in \eqref{SU1} to obtain
 \begin{equation}\label{18-10-22-4}
    \left(\mathbb{D}^{\alpha}_{t}u_{h}^{n},u_{h}^{n}\right)+M\left(\|\nabla \tilde{u}_{h}^{n-1}\|^{2}\right)\left(\nabla u_{h}^{n},\nabla u_{h}^{n}\right)=\left(f_{h}^{n},u_{h}^{n}\right)+\sum_{j=1}^{n-1}\tau_{j+1}\bar{B}\left(t_{n},t_{j},u_{h}^{j},u_{h}^{n}\right).
    \end{equation}
 Apply estimate \eqref{derivative positivity} and  positivity of diffusion coefficient (A2) together with Cauchy-Schwarz inequality and continuity of memory operator  \eqref{2.3}  to get 
 \begin{equation}\label{hg}
     \frac{1}{2}\mathbb{D}^{\alpha}_{t}\|u_{h}^{n}\|^{2}+m_{0}\|\nabla u_{h}^{n}\|^{2}\leq \|f_{h}^{n}\|\|u_{h}^{n}\|+\sum_{j=1}^{n-1}\tau_{j+1}B_{0}\|\nabla u_{h}^{j}\|\|\nabla u_{h}^{n}\|.
 \end{equation}
 Since the equation \eqref{hg} is true for all $n,~(2\leq n \leq N)$, therefore it is also true for $m,~(2\leq m \leq N)$  such that relation  \eqref{18-10-22-3} holds. Thus by the definition of $\mathbb{D}^{\alpha}_{t}\|u_{h}^{n}\|^{2}$  \eqref{approximation of Caputo Derivative} and  weighted norm \eqref{new norm} we have 
 \begin{equation}\label{hg-1}
 \begin{aligned}
     \||u_{h}^{m}\||^{2}&\lesssim\left( \tau_{m}^{\alpha}\|f_{h}^{m}\|+\tau_{m}^{\alpha/2}\sum_{j=1}^{m-1}\tau_{j+1}\tau_{j}^{-\alpha/2}\|| u_{h}^{j}\||+\tau_{m}^{\alpha}\sum_{j=0}^{m-1}(k_{m,j+1}-k_{m,j})\||u_{h}^{j}\||\right)\||u_{h}^{m}\||.
     \end{aligned}
 \end{equation}
 Cancelling $\||u_{h}^{m}\||$ from both sides of \eqref{hg-1} and using discrete Gr\"{o}nwall's inequality, we get
 \begin{equation}\label{hg-2}
 \begin{aligned}
     \||u_{h}^{m}\||&\lesssim\left( \tau_{m}^{\alpha}\|f_{h}^{m}\|+\tau_{m}^{\alpha}k_{m,1}\||u_{h}^{0}\||\right)exp\left(\sum_{j=1}^{m-1}w_{m,j}\right).
     \end{aligned}
 \end{equation}
 Finally, estimates \eqref{tau1kn1 estimate }, \eqref{weight estimate}, and $\tau_{m}^{\alpha} \lesssim t_{m}^{\alpha}\lesssim T^{\alpha}$ imply 
 \begin{equation}\label{df}
     \max_{1\leq n \leq N}\||u_{h}^{n}\||\lesssim \|u_{h}^{0}\|+\max_{1\leq n \leq N}\|f_{h}^{n}\|\lesssim \|\nabla u_{0}\|+\max_{1\leq n \leq N}\|f^{n}\|.
 \end{equation}
\end{proof}
\noindent The equation \eqref{16-9-22-1} has a  gradient type nonlinearity, so we need to obtain gradient bound on numerical solutions of the scheme \eqref{SU1}-\eqref{SUM3}. The estimate \eqref{hg-2} yields 
\begin{equation}\label{poi}
    \|\nabla u_{h}^{m}\|\lesssim \tau_{m}^{\frac{-\alpha}{2}}\|u_{0}\|+\|f^{m}\|.
\end{equation}
From the estimate \eqref{poi}, it is observed that gradient is bounded only when initial condition is zero. To avoid this situation, we make use of  definitions of discrete Laplacian operators \eqref{discrete laplace} and \eqref{discrete laplace-1}. 

\begin{thm}\label{12345-2}
Suppose that $(A1)$-$(A2)$ hold. Then  numerical solutions $u_{h}^{n},(1\leq n \leq N)$ of the scheme \eqref{SU1}-\eqref{SUM3} satisfy the following a priori bound
\begin{equation}\label{grad bound}
    \max_{1\leq n \leq N}\|\nabla u_{h}^{n}\|\lesssim \|\nabla u_{0}\|+\max_{1\leq n \leq N}\| f^{n}\|.
\end{equation}
\end{thm}

\begin{proof} 
Consider the case $n=1$ from the numerical scheme \eqref{SU1}-\eqref{SUM3}, we have
\begin{equation}\label{18-10-22-5}
     \left(\mathbb{D}^{\alpha}_{t}u_{h}^{1},v_{h}\right)+M\left(\|\nabla u_{h}^{1}\|^{2}\right)\left(\nabla u_{h}^{1},\nabla v_{h}\right)=\left(f_{h}^{1},v_{h}\right)+\tau_{1}B\left(t_{1},t_{1},u_{h}^{1},v_{h}\right).
    \end{equation}
    Apply  definitions of discrete Laplacian operators $\Delta_{h}$ \eqref{discrete laplace} and $\Delta_{h}^{b_{2}}$ \eqref{discrete laplace-1} to rewrite the equation \eqref{18-10-22-5} as 
    \begin{equation}\label{18-10-22-6}
    \begin{aligned}
     \left(\mathbb{D}^{\alpha}_{t}u_{h}^{1},v_{h}\right)&+M\left(\|\nabla u_{h}^{1}\|^{2}\right)\left(-\Delta_{h} u_{h}^{1},v_{h}\right)\\
     &=\left(f_{h}^{1},v_{h}\right)+\tau_{1}(-\Delta_{h}^{b_{2}}~u_{h}^{1},v_{h})+\tau_{1}(\nabla \cdot(b_{1}(x,t_{1},t_{1})u_{h}^{1}),v_{h})\\
     &+\tau_{1}(b_{0}(x,t_{1},t_{1})u_{h}^{1},v_{h}).
     \end{aligned}
    \end{equation}
Set $v_{h}=-\Delta_{h} u_{h}^{1}$ in \eqref{18-10-22-6} and apply Cauchy-Schwarz inequality   to obtain
\begin{equation}
\begin{aligned}
    \|\nabla u_{h}^{1}\|^{2}+\tau_{1}^{\alpha}\|\Delta_{h} u_{h}^{1}\|^{2}&\lesssim \|\nabla u_{h}^{0}\|\|\nabla u_{h}^{1}\|+\tau_{1}^{\alpha}\| f_{h}^{1}\|\|\Delta_{h} u_{h}^{1}\|\\
    &+\tau_{1}^{\alpha}\tau_{1}(\|\Delta_{h}^{b_{2}} u_{h}^{1}\|+\|\nabla u_{h}^{1}\|+\|u_{h}^{1}\|)\|\Delta_{h} u_{h}^{1}\|.
    \end{aligned}
\end{equation}
Similar to the estimate \eqref{16-10-22-5}, one have $\|\Delta_{h}^{b_{2}} u_{h}^{1}\| \lesssim \|\nabla u_{h}^{1}\|$ applying which we reach at 
\begin{equation}\label{18-10-22-6-1}
\begin{aligned}
    \|\nabla u_{h}^{1}\|^{2}+\tau_{1}^{\alpha}\|\Delta_{h} u_{h}^{1}\|^{2}&\lesssim \|\nabla u_{h}^{0}\|\|\nabla u_{h}^{1}\|+\tau_{1}^{\alpha}\| f_{h}^{1}\|\|\Delta_{h} u_{h}^{1}\|\\
    &+\tau_{1}^{\alpha}\tau_{1}(\|\nabla u_{h}^{1}\|+\|u_{h}^{1}\|)\|\Delta_{h} u_{h}^{1}\|.
    \end{aligned}
\end{equation}
Parallel to the weighted $H^{1}(\Omega)$ norm, we define a new norm as 
\begin{equation}\label{18-10-22-6-1-1}
   \||u_{h}^{n}\||_{1}:=\|\nabla u_{h}^{n}\|+\tau_{n}^{\alpha/2}\|\Delta_{h} u_{h}^{n}\|~\text{for}~1\leq n \leq N,
\end{equation}
and reduce the equation \eqref{18-10-22-6-1} into 
\begin{equation}\label{18-10-22-7}
\begin{aligned}
    \||u_{h}^{1}\||_{1}^{2}&\lesssim \left(\|\nabla u_{h}^{0}\|+\tau_{1}^{\alpha/2}\| f_{h}^{1}\|+\tau_{1}^{\alpha/2}\tau_{1}\|\nabla u_{h}^{1}\|+\tau_{1}^{\alpha/2}\tau_{1}\|u_{h}^{1}\|\right)\||u_{h}^{1}\||_{1}.
    \end{aligned}
\end{equation}
Cancellation of $\||u_{h}^{1}\||_{1}$ from both sides of \eqref{18-10-22-7} yields 
\begin{equation}\label{18-10-22-8}
\begin{aligned}
    \||u_{h}^{1}\||_{1}&\lesssim \left(\|\nabla u_{h}^{0}\|+\tau_{1}^{\alpha/2}\| f_{h}^{1}\|+\tau_{1}^{\alpha/2}\tau_{1}\|\nabla u_{h}^{1}\|+\tau_{1}^{\alpha/2}\tau_{1}\|u_{h}^{1}\|\right).
    \end{aligned}
\end{equation}
Take sufficiently small $\tau_{1}$ and apply estimate  \eqref {numerical bound} to conclude 
\begin{equation*}
    \|\nabla u_{h}^{1}\| \lesssim \|\nabla u_{h}^{0}\|+\| f_{h}^{1}\|+\||u_{h}^{1}\||  \lesssim \|\nabla u_{h}^{0}\|+ \| f_{h}^{1}\| \lesssim \|\nabla u_{0}\|+\|f^{1}\|.
\end{equation*}
If 
 \begin{equation}\label{18-10-22-2-1}
     \max_{1\leq n \leq N}\||u_{h}^{n}\||_{1}=\||u_{h}^{1}\||_{1},
 \end{equation}
 then proof  of estimate \eqref{grad bound} follows immediately. Otherwise, 
  \begin{equation}\label{18-10-22-3-1}
     \max_{1\leq n \leq N}\||u_{h}^{n}\||_{1}=\||u_{h}^{m}\||_{1}~\text{for some}~m~\in~\{2,3,4,\dots,N\}.
 \end{equation}
 \noindent Now, consider the case  $n\geq 2$ in numerical scheme \eqref{SU1}-\eqref{SUM3}  
 \begin{equation}\label{18-10-22-4-1}
    \left(\mathbb{D}^{\alpha}_{t}u_{h}^{n},v_{h}\right)+M\left(\|\nabla \tilde{u}_{h}^{n-1}\|^{2}\right)\left(\nabla u_{h}^{n},\nabla v_{h}\right)=\left(f_{h}^{n},v_{h}\right)+\sum_{j=1}^{n-1}\tau_{j+1}\bar{B}\left(t_{n},t_{j},u_{h}^{j},v_{h}\right).
    \end{equation}
    Again using definitions of discrete Laplacian operators \eqref{discrete laplace} and \eqref{discrete laplace-1}, we convert the equation \eqref{18-10-22-4-1} into 
    \begin{equation}\label{18-10-22-4-2}
    \begin{aligned}
    \left(\mathbb{D}^{\alpha}_{t}u_{h}^{n},v_{h}\right)&+M\left(\|\nabla \tilde{u}_{h}^{n-1}\|^{2}\right)\left(-\Delta_{h} u_{h}^{n}, v_{h}\right)\\
    &\lesssim \left(f_{h}^{n},v_{h}\right)+\sum_{j=1}^{n-1}\tau_{j+1}(-\Delta_{h}^{\bar{b}_{2}}~u_{h}^{j},v_{h})+\sum_{j=1}^{n-1}\tau_{j+1}\left(\nabla \cdot (\bar{b}_{1}(x,t_{n},t_{j})u_{h}^{j}),v_{h}\right)\\
    &+\sum_{j=1}^{n-1}\tau_{j+1}(\bar{b}_{0}(x,t_{n},t_{j})u_{h}^{j},v_{h}),
    \end{aligned}
    \end{equation}
    here $\bar{b}_{2}, \bar{b}_{1},$ and $\bar{b}_{0}$ has the same meaning as $\bar{B}$ in  \eqref{S9}. Put $v_{h}=-\Delta_{h}u_{h}^{n}$ in \eqref{18-10-22-4-2} to get 
     \begin{equation}\label{18-10-22-4-3}
    \begin{aligned}
   \frac{1}{2}\mathbb{D}^{\alpha}_{t}\|\nabla u_{h}^{n}\|^{2}+\|\Delta_{h} u_{h}^{n}\|^{2}
    &\lesssim\|f_{h}^{n}\|\|\Delta_{h}u_{h}^{n}\|+\sum_{j=1}^{n-1}\tau_{j+1}\|\Delta_{h}^{\bar{b}_{2}}~u_{h}^{j}\|\|\Delta_{h}u_{h}^{n}\|\\
    &+\sum_{j=1}^{n-1}\tau_{j+1}\|\nabla u_{h}^{j}\| \|\Delta_{h}u_{h}^{n}\| +\sum_{j=1}^{n-1}\tau_{j+1}\|u_{h}^{j}\| \|\Delta_{h}u_{h}^{n}\|.
    \end{aligned}
    \end{equation}
 Analogous  to the estimate \eqref{16-10-22-5}, one have $\|\Delta_{h}^{\bar{b}_{2}} u_{h}^{j}\| \lesssim \|\nabla u_{h}^{j}\|$ which leads to 
  \begin{equation}\label{18-10-22-4-4}
    \begin{aligned}
   \frac{1}{2}\mathbb{D}^{\alpha}_{t}\|\nabla u_{h}^{n}\|^{2}+\|\Delta_{h} u_{h}^{n}\|^{2}
    &\lesssim\left(\|f_{h}^{n}\|+\sum_{j=1}^{n-1}\tau_{j+1}\|\nabla u_{h}^{j}\|\right) \|\Delta_{h}u_{h}^{n}\|.
    \end{aligned}
    \end{equation}
    Since the equation \eqref{18-10-22-4-4} is true for all $n\geq 2$, thus it is also true for $m \geq 2$ such that \eqref{18-10-22-3-1} holds. Hence
     \begin{equation}\label{18-10-22-4-5}
    \begin{aligned}
  \|\nabla u_{h}^{m}\|^{2}+\tau_{m}^{\alpha}\|\Delta_{h} u_{h}^{m}\|^{2}
    &\lesssim \tau_{m}^{\alpha}\left(\|f_{h}^{m}\|+\sum_{j=1}^{m-1}\tau_{j+1}\|\nabla u_{h}^{j}\|\right) \|\Delta_{h}u_{h}^{m}\|\\
    &+ \tau_{m}^{\alpha}\sum_{j=0}^{m-1}(k_{m,j+1}-k_{m,j})\|\nabla u_{h}^{j}\|^{2}.
    \end{aligned}
    \end{equation}
    Employ the definition of $\||u_{h}^{m}\||_{1}$ \eqref{18-10-22-6-1-1}  in \eqref{18-10-22-4-5} to arrive at 
     \begin{equation}\label{18-10-22-4-6}
    \begin{aligned}
  \|| u_{h}^{m}\||_{1}^{2}
    &\lesssim \tau_{m}^{\alpha/2}\|f_{h}^{m}\|  \|| u_{h}^{m}\||_{1} +\tau_{m}^{\alpha/2}\sum_{j=1}^{m-1}\tau_{j+1} \tau_{j}^{-\alpha/2} \tau_{j}^{\alpha/2}\|\nabla u_{h}^{j}\| \||u_{h}^{m}\||_{1}\\
    &+ \tau_{m}^{\alpha}\sum_{j=0}^{m-1}(k_{m,j+1}-k_{m,j})\|| u_{h}^{j}\||_{1} \|| u_{h}^{m}\||_{1}.
    \end{aligned}
    \end{equation}
    Cancellation of $\|| u_{h}^{m}\||_{1}$  and discrete Gr\"{o}nwall's inequality reduce \eqref{18-10-22-4-6} to 
    \begin{equation}\label{18-10-22-4-61}
    \begin{aligned}
  \|| u_{h}^{m}\||_{1}
    \lesssim \tau_{m}^{\alpha/2}\|f_{h}^{m}\|+\left(\max_{1\leq j \leq N}\||u_{h}^{j}\||\right)\tau_{m}^{\alpha/2}\sum_{j=1}^{m-1}\tau_{j+1} \tau_{j}^{-\alpha/2}.
    \end{aligned}
    \end{equation}
    Finally, estimates \eqref{numerical bound} and \eqref{weight estimate} imply
    \begin{equation}\label{18-10-22-4-7}
    \begin{aligned}
  \|| u_{h}^{m}\||_{1}
    \lesssim  \|\nabla u_{0}\| + \max_{1\leq n \leq N}\|f^{n}\|.
    \end{aligned}
    \end{equation}

\end{proof}
\subsection{Existence and uniqueness of numerical solution }
\noindent In this subsection, we state an application of Br\"{o}uwer fixed point theorem and as a consequence we prove the existence and uniqueness of numerical solution of the  scheme \eqref{SU1}-\eqref{SUM3}.  
\begin{thm} \cite{thomee2007galerkin}
If $\mathbb{H}$ is a    finite dimensional Hilbert space and  $G:\mathbb{H}\rightarrow \mathbb{H}$ be a continuous map such that $\left(G(\hat{w}),\hat{w}\right)>0$  for all $\hat{w}$ in $\mathbb{H}$ with $\|\hat{w}\|=r,~r>0$ then there exists a $\tilde{w}$ in $\mathbb{H}$ such that $G(\tilde{w})=0$ and $\|\tilde{w}\|\leq r.$  
\end{thm} 

\begin{thm} Under the assumptions $(A1)$-$(A2)$, there exists a unique  solution to the problem \eqref{SU1}-\eqref{SUM3}.
\end{thm}
\begin{proof} For $n\geq 2$, the scheme \eqref{SU1}-\eqref{SUM3} is linear having positive definite coefficient matrix. As a result, existence and uniqueness of solution $u_{h}^{n}~(n\geq 2)$ follow immediately. For the existence of $u_{h}^{1}$ in \eqref{SUM1}, put the value of $\mathbb{D}^{\alpha}_{t}u_{h}^{1}$ from \eqref{approximation of Caputo Derivative} in \eqref{SUM1} to get 
\begin{equation}\label{29-11-22-1}
\begin{aligned}
 \left(\frac{\tau_{1}^{-\alpha}}{\gamma_{0}}\left(u_{h}^{1}-u_{h}^{0}\right),v_{h}\right)+M\left(\|\nabla u_{h}^{1}\|^{2}\right)\left(\nabla u_{h}^{1},\nabla v_{h}\right)&=\left(f_{h}^{1},v_{h}\right)+\tau_{1}B\left(t_{1},t_{1},u_{h}^{1},v_{h}\right).
 \end{aligned}
 \end{equation}
 In the view of \eqref{29-11-22-1},  define a map $G:X_{h}\rightarrow X_{h}$ such that 
\begin{equation*}
\begin{aligned}
\left(G\left(u_{h}^{1}\right),v_{h}\right)&=\left(u_{h}^{1},v_{h}\right)-\left(u_{h}^{0},v_{h}\right)+\tau_{1}^{\alpha}\gamma_{0}M\left(\|\nabla u_{h}^{1}\|^{2}\right)\left(\nabla u_{h}^{1},\nabla v_{h}\right)\\
&-\tau_{1}^{\alpha}\gamma_{0}\left(f_{h}^{1},v_{h}\right)-\tau_{1}\tau_{1}^{\alpha}\gamma_{0}B\left(t_{1},t_{1},u_{h}^{1},v_{h}\right).
\end{aligned}
\end{equation*}
Apply positivity of diffusion coefficient (A2), continuity of memory operator \eqref{2.3}, and Cauchy-Schwarz inequality to obtain
\begin{equation*}
\begin{aligned}
\left(G\left(u_{h}^{1}\right),u_{h}^{1}\right)&\geq \|u_{h}^{1}\|^{2}-\|u_{h}^{0}\|\|u_{h}^{1}\|+\tau_{1}^{\alpha}\gamma_{0}m_{0}\|\nabla u_{h}^{1}\|^{2}-\tau_{1}^{\alpha}\gamma_{0}\|f_{h}^{1}\|\| u_{h}^{1}\|\\
&-\tau_{1}^{1+\alpha}\gamma_{0}B_{0}\|\nabla u_{h}^{1}\|^{2}.
\end{aligned}
\end{equation*}
So 
\begin{equation*}
\begin{aligned}
\left(G\left(u_{h}^{1}\right),u_{h}^{1}\right)&\geq \|u_{h}^{1}\|\left(\|u_{h}^{1}\|-\|u_{h}^{0}\|-\tau_{1}^{\alpha}\gamma_{0}\|f_{h}^{1}\|\right)+\tau_{1}^{\alpha}\gamma_{0}(m_{0}-\tau_{1}B_{0})\|\nabla u_{h}^{1}\|^{2}.
\end{aligned}
\end{equation*}
For $\tau_{1}=\frac{m_{0}}{2B_{0}}$, choose $\|u_{h}^{1}\|$ sufficiently large such that $\left(G\left(u_{h}^{1}\right),u_{h}^{1}\right)>0$ and  $G$ is a continuous map because of continuity of the diffusion coefficient $M$.  So the above stated Br\"{o}uwer fixed point theorem ensures the existence of $u_{h}^{1}$.
 \par Let $X_{h}^{1}$ and $Y_{h}^{1}$ be solutions of the equation \eqref{SUM1}, then $Z_{h}^{1}=X_{h}^{1}-Y_{h}^{1}$ satisfies the following equation for all $v_{h} \in X_{h}$
 \begin{equation}\label{17-11-22-1}
\begin{aligned}
     &\left(Z_{h}^{1},v_{h}\right)+\tau_{1}^{\alpha}\gamma_{0}M\left(\|\nabla X_{h}^{1}\|^{2}\right)\left(\nabla Z_{h}^{1},\nabla v_{h}\right)\\
     &=\tau_{1}^{\alpha}\gamma_{0}\left(M\left(\|\nabla Y_{h}^{1}\|^{2}\right)-M\left(\|\nabla X_{h}^{1}\|^{2}\right)\right)\left(\nabla Y_{h}^{1},\nabla v_{h}\right)+\tau_{1}^{1+\alpha}\gamma_{0}B\left(t_{1},t_{1},Z_{h}^{1},v_{h}\right).
     \end{aligned}
 \end{equation}
 Set $v_{h}=Z_{h}^{1}$ in \eqref{17-11-22-1} and use positivity, Lipschitz continuity of $M$ (A2) along with continuity of the  memory operator \eqref{2.3} and Cauchy-Schwarz inequality to get
  \begin{equation}\label{17-11-22-1-1}
\begin{aligned}
     \|Z_{h}^{1}\|^{2}+\tau_{1}^{\alpha}m_{0}\|\nabla Z_{h}^{1}\|^{2}\lesssim \tau_{1}^{\alpha}L_{M}(\|\nabla X_{h}^{1}\|+\|\nabla Y_{h}^{1}\|)\|\nabla Z_{h}^{1}\|^{2}\|\nabla Y_{h}^{1}\|+\tau_{1}^{1+\alpha}B_{0}\|\nabla Z_{h}^{1}\|^{2}.
     \end{aligned}
 \end{equation}
 Employing the gradient bound \eqref{grad bound} in \eqref{17-11-22-1-1}, we obtain
  \begin{equation*}
\begin{aligned}
     \|Z_{h}^{1}\|^{2}&+\tau_{1}^{\alpha}\left(m_{0}-B_{0}\tau_{1}-2L_{M}K^{2}\right)\|\nabla Z_{h}^{1}\|^{2}\leq 0.
     \end{aligned}
 \end{equation*}
Take $ \tau_{1}=\frac{m_{0}}{2B_{0}}$ and  apply (A2) to deduce
 \begin{equation*}
\begin{aligned}
     \|Z_{h}^{1}\|^{2}+\tau_{1}^{\alpha}\|\nabla Z_{h}^{1}\|^{2}
     &\leq 0.
     \end{aligned}
 \end{equation*}
Thus uniqueness follows.
 \end{proof}
 \subsection{Local truncation errors}
 In this subsection, we derive local truncation errors that arise due to the approximation of Caputo fractional derivative, nonlinear diffusion coefficient, and the memory term. 
 \begin{lem}\textbf{(Approximation of the Caputo fractional derivative \cite{kopteva2019error})}\label{6-5-22-5}
The truncation error $\mathbb{T}^{n}$ at any time $t_{n}\in [0,T]$ arising due to the L1 approximation of Caputo fractional derivative \eqref{p3} satisfies the following estimate 
\begin{equation}\label{p1}
    \left|\mathbb{T}^{n}\right|\lesssim t_{n}^{-\alpha}\left\{\left(\frac{\tau_{1}}{t_{n}}\right)\psi^{1}+\max_{j=2,\dots,n}\psi^{j}\right\},~~n=1,2,3,\dots,N,
\end{equation}
where 
\begin{equation*}
    \psi^{1}=\tau^{\alpha}_{1}\sup_{s \in (0,t_{1})}\left(s^{1-\alpha}\left|\bar{\partial}_{t}u^{1}-\frac{\partial u}{\partial s}(s)\right|\right),~ \text{with}~ \bar{\partial}_{t}u^{1}=\frac{u(t_{1})-u(t_{0})}{\tau_{1}},
\end{equation*}
 and 
\begin{equation*}
   \psi^{j}=\tau_{j}^{2-\alpha}t_{j}^{\alpha}\sup_{s\in(t_{j-1},t_{j})}\left|\frac{\partial^{2}u}{\partial s^{2}}(s)\right|~~\text{for}~~j \geq 2.
\end{equation*}
\end{lem}
\begin{cor}\cite{kopteva2019error} Suppose that the solution $u$ of the problem  \eqref{nonlinear weak problem} satisfies the limited regularity \eqref{14-9-22-2} i.e. $|\partial_{t}^{l}u|\lesssim t^{\alpha-l}$ for $l=1,2$ and $\psi^{j}$ be as in Lemma \ref{6-5-22-5}. Then 
\begin{equation}\label{Sw}
    \psi^{j}\lesssim N^{-\min\{\delta \alpha,~2-\alpha\}},~~\text{for}~~j=1,2,3,\dots,N.
\end{equation}
\end{cor}
\begin{lem}
For any  $n \in \{1,2,3,\dots,N\}$. Let $u^{n}$ be the value of $u$ at $t_{n}$ and $\tilde{u}^{n-1}$ be the approximation of $u$ defined in \eqref{linearization} at $t_{n}$ then the linearization error $\mathbb{L}^{n}$ at $t_{n}$ satisfies
\begin{equation}\label{mk}
    |\mathbb{L}^{n}|=|u^{n}-\tilde{u}^{n-1}|\lesssim N^{-\min\{\delta \alpha,~2\}}.
\end{equation}
\end{lem}
\begin{proof}
Denote $\tilde{\tau}_{n}:=\frac{\tau_{n}}{\tau_{n-1}}$ and consider $\mathbb{L}^{n}=u^{n}-\tilde{u}^{n-1}$. Then by the Taylor series expansion of $u^{n}$ around $u^{n-1}$ and $u^{n-2}$  there exist $\xi_{1} \in (t_{n-1},t_{n})$ and $\xi_{2}\in (t_{n-2},t_{n})$ such that 
\begin{equation*}
\begin{aligned}
    \mathbb{L}^{n}&=\left[(1+\tilde{\tau}_{n})\tau_{n}-\tilde{\tau}_{n}\left(\tau_{n-1}+\tau_{n}\right)\right]\frac{\partial u}{\partial t}(t_{n})+(1+\tilde{\tau}_{n})\frac{{\tau_{n}}^{2}}{2}\frac{\partial^{2}u}{\partial t^{2}}(\xi_{1})\\
    &-\tilde{\tau}_{n}\frac{\left(\tau_{n-1}+\tau_{n}\right)^{2}}{2}\frac{\partial^{2}u}{\partial t^{2}}(\xi_{2}).
    \end{aligned}
\end{equation*}
Since $\frac{\partial^{2}u}{\partial t^{2}}$ is continuous so by intermediate value theorem \cite[Theorem 6.14]{baskar2016introduction}  there exist a $\xi \in (t_{n-2},t_{n})$ such that  
\begin{equation}\label{mk-1}
\begin{aligned}
    \mathbb{L}^{n}&=\left[\frac{{\tau_{n}}^{2}\tilde{\tau}_{n}}{2}+\frac{{\tau_{n}}^{2}}{2}-\frac{\tilde{\tau}_{n}\tau_{n-1}^{2}}{2}-\frac{\tilde{\tau}_{n}\tau_{n}^{2}}{2}-\tau_{n-1}\tau_{n}\tilde{\tau}_{n}\right]\frac{\partial^{2}u}{\partial t^{2}}(\xi).
    \end{aligned}
\end{equation}
Simplification of \eqref{mk-1}  reduces to 
\begin{equation*}
\begin{aligned}
    |\mathbb{L}^{n}|\lesssim \tau_{n}^{2}\left|\frac{\partial^{2}u}{\partial t^{2}}(\xi)\right|.
    \end{aligned}
\end{equation*}
Further, regularity assumption \eqref{14-9-22-2} and relation \eqref{29-11-22-2} for $n\geq 3$ yield
\begin{equation*}
|\mathbb{L}^{n}|\lesssim \tau_{n}^{2}\xi^{\alpha-2}\lesssim \tau_{n}^{2}t_{n-2}^{\alpha-2}\lesssim \tau_{n}^{2}t_{n}^{\alpha-2}.
\end{equation*}
Graded mesh property \eqref{am1} and definition of graded mesh \eqref{SU4} imply
\begin{equation}\label{SU2}
   |\mathbb{L}^{n}|  \lesssim N^{-2\delta}n^{2(\delta-1)}t_{n}^{\alpha-2}\lesssim N^{-2\delta}n^{2(\delta-1)}n^{\delta(\alpha-2)}N^{-\delta(\alpha-2)}\lesssim n^{\delta \alpha-2}N^{-\delta \alpha}~\text{for}~n\geq 3.
\end{equation}
For $n=2$, we apply  Taylor's theorem with integral reminder to get 
\begin{equation}\label{27-5-4-1}
    \left|\mathbb{L}^{2}\right|\lesssim |\int_{t_{2}}^{t_{1}}(s-t_{1})\frac{\partial^{2}u}{\partial t^{2}}(s)~ds+\int_{t_{0}}^{t_{2}}(s)\frac{\partial^{2}u}{\partial t^{2}}(s)~ds|.
\end{equation}
Apply the regularity  estimate \eqref{14-9-22-2} in \eqref{27-5-4-1} to have 
\begin{equation}\label{27-5-4-2}
    \left|\mathbb{L}^{2}\right|\lesssim\int_{t_{1}}^{t_{2}}(s-t_{1})s^{\alpha-2}~ds+\int_{t_{0}}^{t_{2}}(s)s^{\alpha-2}~ds.
\end{equation}
Simplifying equation \eqref{27-5-4-2}  to conclude  $|\mathbb{L}^{2}|\lesssim N^{-\delta \alpha}$.
 From the estimate \eqref{SU2}, it is observed that if $\delta \alpha < 2$ then $|\mathbb{L}^{n}| \lesssim N^{-\delta \alpha}$ and if  $\delta \alpha \geq 2$ then $|\mathbb{L}^{n}| \lesssim N^{-2}.$ Hence result \eqref{mk} follows.
 \end{proof}
\begin{rmk}
Suppose $\delta=1$, i.e., uniform mesh $(\tau_{n}=\tau_{n-1})$ then $\tilde{u}^{n-1}=2u^{n-1}-u^{n-2}$, which is same as the approximation of nonlinear term  proposed in \cite{lalit}.
\end{rmk}
\begin{lem}
Let $\mathbb{Q}^{n}$ be the quadrature error operator on $H^{1}_{0}(\Omega)$ at $t_{n},n\geq 2$ that is given  by
\begin{equation}
    \left(\mathbb{Q}^{n},v\right)=\sum_{j=1}^{n-1}\tau_{j+1}\bar{B}\left(t_{n},t_{j},u^{j},v\right)-\int_{t_{1}}^{t_{n}}B\left(t_{n},s,u(s),v\right)ds~\forall~ v~\in~H^{1}_{0}(\Omega),
\end{equation}
and $\mathbb{Q}^{1}$ be the quadrature error operator on $H^{1}_{0}(\Omega)$ at $t_{1}$ defined by 
\begin{equation}\label{S1}
    \left(\mathbb{Q}^{1},v\right)=\tau_{1}B\left(t_{1},t_{1},u^{1},v\right)-\int_{t_{0}}^{t_{1}}B\left(t_{1},s,u(s),v\right)ds~\forall~ v~\in~H^{1}_{0}(\Omega),
\end{equation}
then $\mathbb{Q}^{1}$ and $\mathbb{Q}^{n}$ satisfy 
\begin{align}
    |\left(\mathbb{Q}^{1},v\right)|&\lesssim N^{-\delta(1+\alpha)}\|\nabla v\| ~\forall~ v~\in~H^{1}_{0}(\Omega),\label{p10}\\
    |\left(\mathbb{Q}^{n},v\right)|&\lesssim N^{-\min\{\delta(1+\alpha),~2\}}\|\nabla v\|~\forall~ v~\in~H^{1}_{0}(\Omega).\label{p0}
\end{align}
\end{lem}
\begin{proof}
First we estimate $\left(\mathbb{Q}^{1},v\right)$.  For $s \in (0,t_{1})$, apply Taylor series expansion of the function $B(t_{1},s,u(s),v)$ around  $B(t_{1},t_{1},u^{1},v)$ to get 
\begin{equation*}
\begin{aligned}
    &B(t_{1},t_{1},u^{1},v)-B(t_{1},s,u(s),v)=(t_{1}-s)\frac{\partial}{\partial s}\left[B\left(t_{1},s,u(s),v\right)\right](\xi_{1}(s))~\text{f.s.}~\xi_{1}(s)\in (s,t_{1})\\
    &=(t_{1}-s)\left(\frac{\partial B}{\partial s}\left(t_{1},s,u(s),v\right)(\xi_{1}(s))+\left[B\left(t_{1},s,\frac{\partial u}{\partial s}(s),v\right)\right](\xi_{1}(s))\right).
    \end{aligned}
\end{equation*}
Integrate both sides over $(t_{0},t_{1})$ and smoothness of  coefficients of the memory operator \eqref{2.3} implies
\begin{equation*}
\begin{aligned}
   |\left(\mathbb{Q}^{1},v\right)|  \lesssim \int_{0}^{t_{1}}(t_{1}-s)\left(\|\nabla u(\xi_{1}(s))\|+\|\nabla \frac{\partial u}{\partial s}(\xi_{1}(s))\|\right)\|\nabla v\|~ds.
    \end{aligned}
\end{equation*}
Using regularity assumptions \eqref{14-9-22-2} and definition of graded mesh \eqref{SU4} implies 
\begin{equation*}
\begin{aligned}
    |\left(\mathbb{Q}^{1},v\right)|  \lesssim  \|\nabla v\| \int_{t_{0}}^{t_{1}}(t_{1}-s)s^{\alpha-1}ds \lesssim \|\nabla v\|t_{1}^{1+\alpha}\lesssim N^{-\delta(1+\alpha)}\|\nabla v\|.
    \end{aligned}
\end{equation*}
 Now, we estimate $\left(\mathbb{Q}^{n},v\right)~n\geq 2$. Consider 
\begin{equation}\label{S8}
\begin{aligned}
    \int_{t_{1}}^{t_{n}}B\left(t_{n},s,u(s),v\right)ds&=\sum_{j=1}^{n-1}\int_{t_{j}}^{t_{j+1}}B\left(t_{n},s,u(s),v\right)ds\\
    &=\sum_{j=1}^{n-2}\int_{t_{j}}^{t_{j+1}}B\left(t_{n},s,u(s),v\right)ds+\int_{t_{n-1}}^{t_{n}}B\left(t_{n},s,u(s),v\right)ds.
    \end{aligned}
\end{equation}
    Let $B^{I}\left(t_{n},s,u(s),v\right)$ be the linear interpolation of $B\left(t_{n},s,u(s),v\right)$  on $[t_{j},t_{j+1}]$ for $ 1\leq j \leq n-2$ in the first part of \eqref{S8} and the second part of \eqref{S8} is approximated by the left rectangle rule then we get 
    \begin{equation}\label{29-11-22-3}
        \begin{aligned}
    \int_{t_{1}}^{t_{n}}B\left(t_{n},s,u(s),v\right)ds
    &=\sum_{j=1}^{n-2}\tau_{j+1}\frac{1}{2}\left[B\left(t_{n},t_{j},u^{j},v\right)+B\left(t_{n},t_{j+1},u^{j+1},v\right)\right]\\
    &+\tau_{n}B\left(t_{n},t_{n-1},u^{n-1},v\right)\\
    &+\sum_{j=1}^{n-2}\int_{t_{j}}^{t_{j+1}}\left(s-t_{j}\right)(s-t_{j+1})\frac{\partial^{2}}{\partial s^{2}}\left[B\left(t_{n},s,u(s),v\right)\right](\eta_{j}(s))ds\\
    &+\int_{t_{n-1}}^{t_{n}}\left(s-t_{n-1}\right)\frac{\partial}{\partial s}\left[B\left(t_{n},s,u(s),v\right)\right](\xi(s))ds\\
    &~\text{for some}~ \eta_{j}(s)\in \left(t_{j},t_{j+1}\right), \xi (s)\in \left(t_{n-1},t_{n}\right).
    \end{aligned}
    \end{equation}
    Use \eqref{S9} in \eqref{29-11-22-3} to obtain 
     \begin{equation}
        \begin{aligned}
    \int_{t_{1}}^{t_{n}}B\left(t_{n},s,u(s),v\right)ds
    &=\sum_{j=1}^{n-1}\tau_{j+1}\bar{B}\left(t_{n},t_{j},u^{j},v\right)\\
    &+\sum_{j=1}^{n-2}\int_{t_{j}}^{t_{j+1}}\left(s-t_{j}\right)(s-t_{j+1})\frac{\partial^{2}}{\partial s^{2}}\left[B\left(t_{n},s,u(s),v\right)\right](\eta_{j}(s))ds\\
    &+\int_{t_{n-1}}^{t_{n}}\left(s-t_{n-1}\right)\frac{\partial}{\partial s}\left[B\left(t_{n},s,u(s),v\right)\right](\xi(s))ds\\
    &~\text{for some}~ \eta_{j}(s)\in \left(t_{j},t_{j+1}\right), \xi(s)\in \left(t_{n-1},t_{n}\right).
    \end{aligned}
    \end{equation}
    Thus, we have 
\begin{equation*}
        \begin{aligned}
    \left(\mathbb{Q}^{n},v\right)&=-\sum_{j=1}^{n-2}\int_{t_{j}}^{t_{j+1}}\left(s-t_{j}\right)(s-t_{j+1})\frac{\partial^{2}}{\partial s^{2}}\left[B\left(t_{n},s,u(s),v\right)\right](\eta_{j}(s))ds\\
    &-\int_{t_{n-1}}^{t_{n}}\left(s-t_{n-1}\right)\frac{\partial}{\partial s}\left[B\left(t_{n},s,u(s),v\right)\right](\xi(s))ds\\
    &~\text{for some}~ \eta_{j}(s)\in \left(t_{j},t_{j+1}\right), \xi(s) \in \left(t_{n-1},t_{n}\right).
    \end{aligned}
    \end{equation*}
    By mean value theorem for integral (\cite{burden2015numerical}, Theorem 1.13), there exists $\eta_{j}\in (t_{j},t_{j+1})$ and $\xi \in (t_{n-1},t_{n})$ such that 
    \begin{equation}\label{S12}
        \begin{aligned}
    |\left(\mathbb{Q}^{n},v\right)|&\leq \sum_{j=1}^{n-2}\max_{\eta_{j}\in \left(t_{j},t_{j+1}\right)}\left|\frac{\partial^{2}}{\partial s^{2}}\left[B\left(t_{n},s,u(s),v\right)\right](\eta_{j})\right|\frac{\tau_{j+1}^{3}}{6}\\
    &+\max_{\xi\in \left(t_{n-1},t_{n}\right)}\left|\frac{\partial}{\partial s}\left[B\left(t_{n},s,u(s),v\right)\right](\xi)\right|\frac{\tau_{n}^{2}}{2},
    \end{aligned}
    \end{equation}
    where 
    \begin{equation*}
    \begin{aligned}
        \frac{\partial^{2}}{\partial s^{2}}\left[B\left(t_{n},s,u(s),v\right)\right](\eta_{j})&=\frac{\partial^{2}B}{\partial s^{2}}\left(t_{n},s,u(s),v\right)(\eta_{j})+2\frac{\partial B}{\partial s}\left(t_{n},s,\frac{\partial u}{\partial s}(s),v\right)(\eta_{j})\\
        &+B\left(t_{n},s,\frac{\partial ^{2}u}{\partial s^{2}}(s),v\right)(\eta_{j}),
        \end{aligned}
    \end{equation*}
    with 
    \begin{equation*}
\begin{aligned}
    \frac{\partial^{2} B}{\partial s^{2}}\left(t_{n},s,u(s),v\right)&=\left(\frac{\partial^{2}}{\partial s^{2}}(b_{2}(x,t_{n},s))\nabla u(s),\nabla v\right)+\left(\frac{\partial^{2}}{\partial s^{2}}(b_{1}(x,t_{n},s))\nabla u(s), v\right)\\
    &+\left(\frac{\partial^{2}}{\partial s^{2}}(b_{0}(x,t_{n},s)) u(s), v\right).
    \end{aligned}
\end{equation*}
Smoothness of coefficients of the  memory operator \eqref{2.3} and  regularity assumption \eqref{14-9-22-2} imply
\begin{equation}\label{S10}
    \begin{aligned}
        \left|\frac{\partial^{2}}{\partial s^{2}}\left[B\left(t_{n},s,u(s),v\right)\right](\eta_{j})\right|&\lesssim  \left(\|\nabla u(\eta_{j})\|+\|\nabla \frac{\partial u}{\partial s}(s)(\eta_{j})\|+\|\nabla \frac{\partial^{2}u}{\partial s^{2}}(\eta_{j})\|\right)\|\nabla v\|\\
        &\lesssim\left[C_{0}+\eta_{j}^{\alpha-1}+\eta_{j}^{\alpha-2}\right]\|\nabla v\|.
        \end{aligned}
    \end{equation}
    Similarly, one have  
    \begin{equation}\label{S13}
        \left|\frac{\partial}{\partial s}\left[B\left(t_{n},s,u(s),v\right)\right](\xi)\right|\lesssim \left[C_{0}+\xi^{\alpha-1}\right]\|\nabla v\|.
    \end{equation}
    In the view of estimates \eqref{S10} and \eqref{S13} in \eqref{S12}, we get
    \begin{equation*}\label{S14}
        \begin{aligned}
    &|\left(\mathbb{Q}^{n},v\right)|\\
    &\lesssim \sum_{j=1}^{n-2}\tau_{j+1}^{3}\max_{\eta_{j}\in \left(t_{j},t_{j+1}\right)}\left[C_{0}+\eta_{j}^{\alpha-1}+\eta_{j}^{\alpha-2}\right]\|\nabla v\|+\max_{\xi\in \left(t_{n-1},t_{n}\right)}\tau_{n}^{2}\left[C_{0}+\xi^{\alpha-1}\right]\|\nabla v\|\\
    &\lesssim \sum_{j=1}^{n-2}\tau_{j+1}^{3}\max_{\eta_{j}\in \left(t_{j},t_{j+1}\right)}\left[C_{0}\eta_{j}^{\alpha-2}\eta_{j}^{2-\alpha}+\eta_{j}^{\alpha-1}\eta_{j}^{\alpha-2}\eta_{j}^{2-\alpha}+\eta_{j}^{\alpha-2}\right]\|\nabla v\|\\
    &+\max_{\xi\in \left(t_{n-1},t_{n}\right)}\tau_{n}^{2}\left[C_{0}\xi^{\alpha-1}\xi^{1-\alpha}+\xi^{\alpha-1}\right]\|\nabla v\|.
    \end{aligned}
    \end{equation*}
    Since $\eta_{j}^{\alpha-2}\leq t_{j}^{\alpha-2}$ and $\eta_{j}^{2-\alpha}\leq t_{j+1}^{2-\alpha}$ also $\xi^{\alpha-1}\leq t_{n-1}^{\alpha-1}$ and $\xi^{1-\alpha}\leq t_{n}^{1-\alpha}$ so we obtain
    \begin{equation}\label{S15}
        \begin{aligned}
    |\left(\mathbb{Q}^{n},v\right)|
    &\lesssim \sum_{j=1}^{n-2}\tau_{j+1}^{3}t_{j}^{\alpha-2}\|\nabla v\|+\tau_{n}^{2}t_{n-1}^{\alpha-1}\|\nabla v\|.
    \end{aligned}
    \end{equation}
    Employing  the definition of $t_{j}$ and estimate \eqref{am1}, we have
    \begin{equation*}
    \begin{aligned}
        \tau_{j+1}^{3}t_{j}^{\alpha-2}&\leq T^{3}N^{-3\delta}j^{3(\delta-1)}T^{\alpha-2}j^{\delta(\alpha-2)}N^{-\delta(\alpha-2)}\\
        &\lesssim  j^{3\delta-3+\delta \alpha-2\delta}N^{-3\delta-\delta\alpha+2\delta}\\
        &\lesssim  j^{\delta(\alpha+1)-3}N^{-\delta(\alpha+1)}.
        \end{aligned}
    \end{equation*}
    Now, apply the well-known convergence results \cite{stynes2017error} for the series to reach at
\begin{equation}\label{1a}
    \sum_{j=1}^{\left[\frac{n}{2}\right]-1}j^{\delta(\alpha+1)-3}n^{-\delta(\alpha+1)}\leq \begin{cases} n^{-\delta(\alpha+1)}&~ \text{if}~\delta(\alpha+1)<2,\\
    n^{-2}\ln n &~\text{if}~\delta(\alpha+1)=2,\\
    n^{-2}&~\text{if}~\delta(\alpha+1)>2.
    \end{cases}
\end{equation} 
Further $\left[\frac{n}{2}\right]\leq j \leq n-2$, we have 
\begin{equation*}
\begin{aligned}
    &\left|\sum_{j=\left[\frac{n}{2}\right]}^{n-2}\int_{t_{j}}^{t_{j+1}}\left(s-t_{j}\right)\left(s-t_{j+1}\right)\frac{\partial^{2}}{\partial s^{2}}\left[B\left(t_{n},s,u(s),v\right)\right](\eta_{j})ds\right|\\
    &\leq \|\nabla v\|\sum_{j=\left[\frac{n}{2}\right]}^{n-2}\tau_{j+1}^{2}t_{j}^{\alpha-2}\int_{t_{j}}^{t_{j+1}}1~ds.
    \end{aligned}
\end{equation*}
Use estimate \eqref{am1} and $t_{j}^{\alpha-2}\lesssim t_{n}^{\alpha-2}$ for $\left[\frac{n}{2}\right]\leq j \leq n-2$ to get 
\begin{equation}\label{1a-1}
\begin{aligned}
    &\left|\sum_{j=\left[\frac{n}{2}\right]}^{n-2}\int_{t_{j}}^{t_{j+1}}\left(s-t_{j}\right)\left(s-t_{j+1}\right)\frac{\partial^{2}}{\partial s^{2}}\left[B\left(t_{n},s,u(s),v\right)\right](\eta_{j})ds\right|\\
    &\leq T^{2}N^{-2\delta}n^{2(\delta-1) }t_{n}^{\alpha-2}\sum_{j=\left[\frac{n}{2}\right]}^{n-2}\int_{t_{j}}^{t_{j+1}}1~ds\|\nabla v\|\\
    &\lesssim T^{2}N^{-2\delta}n^{2(\delta-1)}t_{n}^{\alpha-2}\int_{t_{\left[\frac{n}{2}\right]}}^{t_{n-1}}1~ds\|\nabla v\|\\
    &\lesssim T^{2}N^{-2\delta}n^{2(\delta-1)}t_{n}^{\alpha-2}t_{n-1}\|\nabla v\|\\
    &\lesssim T^{2}N^{-2\delta}n^{2(\delta-1)}t_{n}^{\alpha-1}\|\nabla v\|.
    \end{aligned}
\end{equation}
 Put the value of $t_{n}$ in \eqref{1a-1} to deduce
\begin{equation}\label{1c}
\begin{aligned}
    &\left|\sum_{j=\left[\frac{n}{2}\right]}^{n-2}\int_{t_{j}}^{t_{j+1}}\left(s-t_{j}\right)\left(s-t_{j+1}\right)\frac{\partial^{2}}{\partial s^{2}}\left[B\left(t_{n},s,u(s),v\right)\right](\eta_{j})ds\right|\\
    &\lesssim T^{2}N^{-2\delta}n^{2(\delta-1)}T^{\alpha-1}n^{\delta(\alpha-1)}N^{-\delta(\alpha-1)}\|\nabla v\|\\
    &\lesssim N^{-\delta(\alpha+1)}n^{\delta(1+\alpha)-2}\|\nabla v\|.
    \end{aligned}
\end{equation}
Similarly, in the last part of \eqref{S15} we have 
\begin{equation}\label{1d}
    \tau_{n}^{2}t_{n-1}^{\alpha-1}\lesssim \tau_{n}^{2}t_{n}^{\alpha-1}\lesssim  N^{-\delta(\alpha+1)}n^{\delta(1+\alpha)-2}.
\end{equation}
Combine the estimates \eqref{1a}-\eqref{1d} in \eqref{S15} to conclude  the result \eqref{p0}. 
\end{proof}
\subsection{Global convergence rates}
\noindent In this subsection, we combine the local truncation errors  \eqref{p1},\eqref{mk},\eqref{p10}, and \eqref{p0} to establish global convergence rates. We decompose the error $(u(t_{n})-u_{h}^{n})$ into  projection error $\left(\rho^{n}:=(u-M_{h}u)(t_{n})\right)$~ and the truncation error $\left(\theta^{n}:=M_{h}u(t_{n})-u_{h}^{n}\right)$  as in the semi discrete error analysis. 

\begin{thm}\label{Main Theorem} Suppose that $(A1)$-$(A2)$ hold and the solution $u$ of the equation \eqref{16-9-22-1} satisfies limited regularity assumptions \eqref{14-9-22-1}-\eqref{14-9-22-2}. Then, the fully discrete scheme \eqref{SU1}-\eqref{SUM3} has the following rate of accuracy 
\begin{equation}\label{lk}
    \max_{1\leq n \leq N}\|u(t_{n})-u_{h}^{n}\|\lesssim P^{-1}+N^{-\min\{\delta \alpha,2-\alpha\}}.
\end{equation}
\end{thm}
\begin{proof} For $n=1$, put $u_{h}^{1}=M_{h}u^{1}-\theta^{1}$ in \eqref{SUM1}. Using weak formulation \eqref{nonlinear weak problem} and definition of modified Ritz-Volterra projection operator \eqref{2.5}, we obtain
\begin{equation}\label{30-11-22-2}
\begin{aligned}
    \left(\mathbb{D}^{\alpha}_{t}\theta^{1},v_{h}\right)&+M\left(\|\nabla u_{h}^{1}\|^{2}\right)(\nabla \theta^{1},\nabla v_{h})\\
    &=\left(\mathbb{D}^{\alpha}_{t}M_{h}u^{1}-~^{C}D^{\alpha}_{t_{1}}M_{h}u,v_{h}\right)+\left(^{C}D^{\alpha}_{t_{1}}M_{h}u-~^{C}D^{\alpha}_{t_{1}}u,v_{h}\right)\\
    &+\left[M\left(\|\nabla u_{h}^{1}\|^{2}\right)-M\left(\|\nabla u^{1}\|^{2}\right)\right](\nabla M_{h}u^{1},\nabla v_{h})+\tau_{1}B(t_{1},t_{1},\theta^{1},v_{h})\\
    &+\int_{0}^{t_{1}}B(t_{1},s,M_{h}u(s),v_{h})~ds-\tau_{1}B(t_{1},t_{1},M_{h}u^{1},v_{h}).
    \end{aligned}
\end{equation}
Set $v_{h}=\theta^{1}$ in \eqref{30-11-22-2} with $\theta^{0}=0$ and apply \eqref{approximation of Caputo Derivative}, (A2) together with \eqref{2.3} and  Cauchy-Schwarz inequality   to get 
\begin{equation}\label{30-11-22-3}
\begin{aligned}
    \|\theta^{1}\|^{2}&+\tau_{1}^{\alpha}(m_{0}-\tau_{1}B_{0})\|\nabla \theta^{1}\|^{2}\\
    &\lesssim\tau_{1}^{\alpha}\left(\|\mathbb{T}^{1}\|+\|~^{C}D^{\alpha}_{t_{1}}\rho\|\right)\|\theta^{1}\|+\tau_{1}^{\alpha}\|\mathbb{Q}^{1}\|\|\nabla \theta^{1}\|\\
    &+\tau_{1}^{\alpha}L_{M}(\|\nabla u_{h}^{1}\|+\|\nabla u^{1}\|)(\|\nabla \rho^{1}\|+\|\nabla \theta^{1}\|)\|\nabla M_{h}u^{1}\|~\|\nabla \theta^{1}\|.
    \end{aligned}
\end{equation}
Employ  a priori bounds \eqref{grad bound}, \eqref{AMP1as}, and \eqref{16-10-22-7} to reduce the equation \eqref{30-11-22-3} into 
\begin{equation}\label{30-11-22-4}
\begin{aligned}
    \|\theta^{1}\|^{2}&+\tau_{1}^{\alpha}(m_{0}-\tau_{1}B_{0}-2L_{M}K^{2})\|\nabla \theta^{1}\|^{2}\\
    &\lesssim\tau_{1}^{\alpha}\left(\|\mathbb{T}^{1}\|+\|~^{C}D^{\alpha}_{t_{1}}\rho\|\right)\|\theta^{1}\|+\tau_{1}^{\alpha}\|\mathbb{Q}^{1}\|\|\nabla \theta^{1}\|+\tau_{1}^{\alpha}\|\nabla \rho^{1}\|~\|\nabla \theta^{1}\|.
    \end{aligned}
\end{equation}
For sufficiently small $\tau_{1}<\frac{m_{0}}{B_{0}}$ for instance take $\tau_{1}=\frac{m_{0}}{2B_{0}}$ to obtain
\begin{equation}\label{30-11-22-5}
\begin{aligned}
    \|\theta^{1}\|^{2}+\tau_{1}^{\alpha}(m_{0}-4L_{M}K^{2})\|\nabla \theta^{1}\|^{2}
    &\lesssim\tau_{1}^{\alpha}\left(\|\mathbb{T}^{1}\|+\|~^{C}D^{\alpha}_{t_{1}}\rho\|\right)\|\theta^{1}\|+\tau_{1}^{\alpha}\|\mathbb{Q}^{1}\|\|\nabla \theta^{1}\|\\
    &+\tau_{1}^{\alpha}\|\nabla \rho^{1}\|~\|\nabla \theta^{1}\|.
    \end{aligned}
\end{equation}
We utilize (A2) and definition of weighted $H^{1}_{0}(\Omega)$ norm \eqref{new norm} to convert the equation \eqref{30-11-22-5} into 
\begin{equation}\label{30-11-22-6}
\begin{aligned}
    \||\theta^{1}\||^{2}
    &\lesssim \tau_{1}^{\alpha}\|\mathbb{T}^{1}\|~\||\theta^{1}\||+ \tau_{1}^{\alpha}\|~^{C}D^{\alpha}_{t_{1}}\rho\|~\||\theta^{1}\||+\tau_{1}^{\alpha/2}\|\mathbb{Q}^{1}\|~\|| \theta^{1}\||+\tau_{1}^{\alpha/2}\|\nabla \rho^{1}\|~\|| \theta^{1}\||.
    \end{aligned}
\end{equation}
Cancellation of $\||\theta^{1}\||$ from both sides of \eqref{30-11-22-6} and approximation properties \eqref{Sw}, \eqref{B10}, \eqref{p10}, \eqref{A10} yield
\begin{equation}\label{30-11-22-7}
\begin{aligned}
    \||\theta^{1}\||
    &\lesssim \tau_{1}^{\alpha}t_{1}^{-\alpha}N^{-\delta\alpha}+ \tau_{1}^{\alpha}h^{2}+\tau_{1}^{\alpha/2}N^{-\delta(1+\alpha)}+\tau_{1}^{\alpha/2}h \lesssim h+N^{-\delta\alpha}.
    \end{aligned}
\end{equation}
If
\begin{equation}
    \max_{1\leq n \leq N}\||\theta^{n}\||=\||\theta^{1}\||,
\end{equation} 
then proof of \eqref{lk} follows immediately. If 
\begin{equation}\label{30-11-22-8}
    \max_{1\leq n \leq N}\||\theta^{n}\||=\||\theta^{m}\||~\text{ for any}~ m~ \in ~\{2,3,4,\dots,N\},
\end{equation} 
then we consider \eqref{SU1} with $u_{h}^{n}=M_{h}u^{n}-\theta^{n}$. Evaluation of  weak formulation \eqref{nonlinear weak problem} and modified Ritz-Volterra projection operator \eqref{2.5} at $t_{n}$ produce the following error equation in $\theta^{n}$
\begin{equation}\label{asd}
\begin{aligned}
    \left(\mathbb{D}^{\alpha}_{t}\theta^{n},v_{h}\right)&+M\left(\|\nabla \tilde{u}_{h}^{n-1}\|^{2}\right)(\nabla \theta^{n},\nabla v_{h})\\
    &=(-\mathbb{T}^{n}-~^{C}D^{\alpha}_{t}\rho^{n}-\mathbb{Q}^{n},v_{h})+\sum_{j=1}^{n-1}\tau_{j+1}\bar{B}(t_{n},t_{j},\theta^{j},v_{h})\\
    &+\left(M\left(\|\nabla \tilde{u}_{h}^{n-1}\|^{2}\right)-M\left(\|\nabla u^{n}\|^{2}\right)\right)(\nabla M_{h}u^{n},\nabla v_{h}).
    \end{aligned}
\end{equation}
Substitute  $v_{h}=\theta^{n}$ in \eqref{asd} and employ \eqref{derivative positivity}, (A2) along  with \eqref{2.3} and  Cauchy-Schwarz inequality to get
\begin{equation*}\label{thetan 2}
\begin{aligned}
    \frac{1}{2}\mathbb{D}^{\alpha}_{t}\|\theta^{n}\|^{2}+m_{0}\|\nabla \theta^{n}\|^{2}&\leq 
    \left(\|\mathbb{T}^{n}\|+\|~^{C}D^{\alpha}_{t}\rho^{n}\|\right)\|\theta^{n}\|+\sum_{j=1}^{n-1}\tau_{j+1}B_{0}\|\nabla \theta^{j}\|\|\nabla \theta^{n}\|\\
    &+L_{M}\left(\|\nabla \tilde{u}_{h}^{n-1}\|+\|\nabla u^{n}\|\right)\left(\|\nabla \tilde{u}_{h}^{n-1}-\nabla u^{n}\|\right)\|\nabla M_{h}u^{n}\|\|\nabla \theta^{n}\|\\
    &+\|\mathbb{Q}^{n}\|\|\nabla \theta^{n}\|.
    \end{aligned}
\end{equation*}
Applying  \eqref{approximation of Caputo Derivative} and (A2), we have
\begin{equation*}\label{thetan 2-1}
\begin{aligned}
     &\|\theta^{n}\|^{2}+\tau_{n}^{\alpha}\|\nabla \theta^{n}\|^{2}\\
     &\lesssim \tau_{n}^{\alpha}
    \left(\|\mathbb{T}^{n}\|+\|~^{C}D^{\alpha}_{t}\rho^{n}\|\right)\|\theta^{n}\|+\tau_{n}^{\alpha}\sum_{j=1}^{n-1}\tau_{j+1}\|\nabla \theta^{j}\|\|\nabla \theta^{n}\|+\tau_{n}^{\alpha}\|\mathbb{Q}^{n}\|\|\nabla \theta^{n}\|\\
    &+\tau_{n}^{\alpha}L_{M}\left(\|\nabla \tilde{u}_{h}^{n-1}\|+\|\nabla u^{n}\|\right)\left(\|\nabla \tilde{\rho}^{n-1}\|+\|\nabla \tilde{\theta}^{n-1}\|+\|\nabla \mathbb{L}^{n}\|\right)\|\nabla M_{h}u^{n}\|\|\nabla \theta^{n}\|\\
    &+\tau_{n}^{\alpha}\sum_{j=0}^{n-1}(k_{n,j+1}-k_{n,j})\|\theta^{j}\|^{2}.
    \end{aligned}
\end{equation*}
A priori bounds \eqref{grad bound}, \eqref{AMP1as}, \eqref{16-10-22-7} and weighted $H^{1}_{0}(\Omega)$ norm \eqref{new norm} imply 
\begin{equation}\label{thetan 4}
\begin{aligned}
    &\||\theta^{n}\||^{2}\\
    &\lesssim 
    \left(\tau_{n}^{\alpha}\|\mathbb{T}^{n}\|+\tau_{n}^{\alpha}\|~^{C}D^{\alpha}_{t}\rho^{n}\|+\tau_{n}^{\alpha/2}\|\mathbb{Q}^{n}\|+\tau_{n}^{\alpha/2}\|\nabla \mathbb{L}^{n}\|+\tau_{n}^{\alpha/2}\|\nabla \tilde{\rho}^{n-1}\|\right)\||\theta^{n}\||\\
    &+\tau_{n}^{\alpha/2}\sum_{j=1}^{n-1}\tau_{j+1}\tau_{j}^{-\alpha/2}\|| \theta^{j}\||~\|| \theta^{n}\||+\tau_{n}^{\alpha/2}\tau_{n-1}^{-\alpha/2}\||\theta^{n-1}\||~\||\theta^{n}\||\\
    &+\tau_{n}^{\alpha/2}\tau_{n-2}^{-\alpha/2}\||\theta^{n-2}\||~\||\theta^{n}\||+\tau_{n}^{\alpha}\sum_{j=0}^{n-1}(k_{n,j+1}-k_{n,j})\||\theta^{j}\||^{2}.
    \end{aligned}
\end{equation}
Since the equation \eqref{thetan 4} is true for all $n~(2\leq n \leq N)$, therefore it is also true for $m~(2\leq m \leq N)$ at which \eqref{30-11-22-8} holds. Thus,
\begin{equation}\label{30-11-22-9}
\begin{aligned}
    &\||\theta^{m}\||^{2}\\
    &\lesssim    \left(\tau_{m}^{\alpha}\|\mathbb{T}^{m}\|+\tau_{m}^{\alpha}\|~^{C}D^{\alpha}_{t}\rho^{m}\|+\tau_{m}^{\alpha/2}\|\mathbb{Q}^{m}\|+\tau_{m}^{\alpha/2}\|\nabla \mathbb{L}^{m}\|+\tau_{m}^{\alpha/2}\|\nabla \tilde{\rho}^{m-1}\|\right)\||\theta^{m}\||\\
    &+\tau_{m}^{\alpha/2}\sum_{j=1}^{m-1}\tau_{j+1}\tau_{j}^{-\alpha/2}\|| \theta^{j}\||~\|| \theta^{m}\||+\tau_{m}^{\alpha/2}\tau_{m-1}^{-\alpha/2}\||\theta^{m-1}\||~\||\theta^{m}\||\\
    &+\tau_{m}^{\alpha/2}\tau_{m-2}^{-\alpha/2}\||\theta^{m-2}\||~\||\theta^{m}\||+\tau_{m}^{\alpha}\sum_{j=0}^{m-1}(k_{m,j+1}-k_{m,j})\||\theta^{j}\||~\||\theta^{m}\||.
    \end{aligned}
\end{equation}
Cancel  $\||\theta^{m}\||$ from both sides of \eqref{30-11-22-9} and apply discrete Gr\"{o}nwall's inequality to get
\begin{equation*}\label{thetan 7}
\begin{aligned}
    \||\theta^{m}\||&\lesssim C_{G}
    \left(\tau_{m}^{\alpha}\|\mathbb{T}^{m}\|+\tau_{m}^{\alpha}\|~^{C}D^{\alpha}_{t}\rho^{m}\|+\tau_{m}^{\alpha/2}\|\mathbb{Q}^{m}\|+\tau_{m}^{\alpha/2}\|\nabla \mathbb{L}^{m}\|+\tau_{m}^{\alpha/2}\|\nabla \tilde{\rho}^{m-1}\|\right),
    \end{aligned}
\end{equation*}
where constant $C_{G}$ is given by 
\begin{equation*}
    C_{G}=exp\left(\tau_{m}^{\alpha/2}\sum_{j=1}^{m-1}\tau_{j+1}\tau_{j}^{-\alpha/2}+\tau_{m}^{\alpha/2}\tau_{m-1}^{-\alpha/2}+\tau_{m}^{\alpha/2}\tau_{m-2}^{-\alpha/2}+\tau_{m}^{\alpha}\sum_{j=0}^{m-1}(k_{m,j+1}-k_{m,j})\right).
\end{equation*}
Making use of Proposition \ref{7-8-1} together with approximation properties  \eqref{p1}, \eqref{B10}, \eqref{p0}, \eqref{mk},  we deduce  
\begin{equation}\label{max est}
\begin{aligned}
    \||\theta^{m}\||&\lesssim \tau_{m}^{\alpha}t_{m}^{-\alpha}N^{-\min\{\delta \alpha,2-\alpha\}}+\tau_{m}^{\alpha}h^{2}+\tau_{m}^{\alpha/2}N^{-\min\{\delta(1+\alpha),~2\}}+\tau_{m}^{\alpha/2}N^{-\min\{\delta \alpha,2\}}+\tau_{m}^{\alpha/2}h\\
    &\lesssim  h+N^{-\min\{\delta \alpha,2-\alpha\}}.
    \end{aligned}
\end{equation}
Finally, by triangle inequality, \eqref{A10}, \eqref{max est}, we conclude 
\begin{equation}\label{30-11-22-10}
\begin{aligned}
    \max_{1\leq n \leq N}\|u(t_{n})-u_{h}^{n}\|&\lesssim \max_{1\leq n \leq N}\| \rho^{n}\|+\max_{1\leq n \leq N}\| \theta^{n}\|\\
    &\lesssim \max_{1\leq n \leq N}\| \rho^{n}\|+\max_{1\leq n \leq N}\|| \theta^{n}\||\\
    &\lesssim  h^2+h+N^{-\min\{\delta \alpha,2-\alpha\}}\lesssim  h+N^{-\min\{\delta \alpha,2-\alpha\}}.
    \end{aligned}
\end{equation}
The result \eqref{lk} follows by taking $h=P^{-1}$, where $P$  is the number of degree of freedom in the space direction.
\end{proof}
\begin{rmk}
If we use the estimate \eqref{max est} to obtain  $L^{\infty}(0,T;H^{1}_{0}(\Omega))$ norm convergence rate, then there is a  loss of accuracy  of order $(\alpha/2)$ as follows
\begin{equation}
    \|\nabla \theta^{m}\|\lesssim \tau_{m}^{-\alpha/2}\left(h+N^{-\min\{\delta \alpha,2-\alpha\}}\right).
\end{equation}
We recover  this loss of accuracy in the following theorem with the help of  discrete Laplacian operators  defined in \eqref{discrete laplace} and \eqref{discrete laplace-1}.
\end{rmk}
\begin{thm}\label{Main Theorem-1} Suppose that $(A1)$-$(A2)$ hold and the solution $u$ of the equation \eqref{16-9-22-1} satisfies limited regularity assumptions \eqref{14-9-22-1}-\eqref{14-9-22-2}. Then, the fully discrete scheme \eqref{SU1}-\eqref{SUM3} has the following rate of accuracy 
\begin{equation}\label{lks}
    \max_{1\leq n \leq N}\|\nabla u(t_{n})-\nabla u_{h}^{n}\|\lesssim P^{-1}+N^{-\min\{\delta \alpha,2-\alpha\}}.
\end{equation}
\end{thm}
\begin{proof}For $n=1$, $\theta^{1}$ satisfies
\begin{equation}\label{30-11-22-11}
\begin{aligned}
    \left(\mathbb{D}^{\alpha}_{t}\theta^{1},v_{h}\right)&+M\left(\|\nabla u_{h}^{1}\|^{2}\right)(\nabla \theta^{1},\nabla v_{h})\\
    &=\left(\mathbb{D}^{\alpha}_{t}M_{h}u^{1}-~^{C}D^{\alpha}_{t_{1}}M_{h}u,v_{h}\right)+\left(^{C}D^{\alpha}_{t_{1}}M_{h}u-~^{C}D^{\alpha}_{t_{1}}u,v_{h}\right)\\
    &+\left[M\left(\|\nabla u_{h}^{1}\|^{2}\right)-M\left(\|\nabla u^{1}\|^{2}\right)\right](\nabla M_{h}u^{1},\nabla v_{h})+\tau_{1}B(t_{1},t_{1},\theta^{1},v_{h})\\
    &+\int_{0}^{t_{1}}B(t_{1},s,M_{h}u(s),v_{h})~ds-\tau_{1}B(t_{1},t_{1},M_{h}u^{1},v_{h}).
    \end{aligned}
\end{equation}
Employ the definition of discrete Laplacian operators \eqref{discrete laplace} and \eqref{discrete laplace-1} in \eqref{30-11-22-11} to get 
\begin{equation}\label{30-11-22-12}
\begin{aligned}
    \left(\mathbb{D}^{\alpha}_{t}\theta^{1},v_{h}\right)&+M\left(\|\nabla u_{h}^{1}\|^{2}\right)(-\Delta_{h} \theta^{1}, v_{h})\\
    &=\left(-\mathbb{T}^{1},v_{h}\right)+\left(-P_{h}~^{C}D^{\alpha}_{t_{1}}\rho,v_{h}\right)+\left(-\mathbb{Q}^{1},v_{h}\right)\\
    &+\left[M\left(\|\nabla u_{h}^{1}\|^{2}\right)-M\left(\|\nabla u^{1}\|^{2}\right)\right](-\Delta_{h} M_{h}u^{1}, v_{h})+\tau_{1}\left(-\Delta_{h}^{b_{2}}~\theta^{1},v_{h}\right)\\
    &+\tau_{1}(\nabla \cdot(b_{1}(x,t_{1},t_{1})\theta^{1}),v_{h})+\tau_{1}(b_{0}(x,t_{1},t_{1})\theta^{1},v_{h}).
    \end{aligned}
\end{equation}
Taking $v_{h}=-\Delta_{h}\theta^{1}$ in \eqref{30-11-22-12}, we have 
\begin{equation}\label{30-11-22-13}
\begin{aligned}
    \|\nabla \theta^{1}\|^{2}+\tau_{1}^{\alpha}\|\Delta_{h} \theta^{1}\|^{2}&\lesssim \tau_{1}^{\alpha}\|\nabla \mathbb{T}^{1}\|\|\nabla \theta^{1}\|+\tau_{1}^{\alpha}\|\nabla P_{h}~^{C}D^{\alpha}_{t_{1}}\rho\|\|\nabla \theta^{1}\|+\tau_{1}^{\alpha}\|\mathbb{Q}^{1}\|\|\Delta_{h}\theta^{1}\|\\
    &+\tau_{1}^{\alpha}\|\nabla \rho^{1}\|\|\Delta_{h}\theta^{1}\|+\tau_{1}^{\alpha}\|\nabla \theta^{1}\|\|\Delta_{h}\theta^{1}\|+\tau^{1+\alpha}\|\Delta_{h}^{b_{2}}~\theta^{1}\|\|\Delta_{h}\theta^{1}\|\\
    &+\tau^{1+\alpha}\|\nabla \theta^{1}\|\|\Delta_{h}\theta^{1}\|+\tau^{1+\alpha}\|\theta^{1}\|\|\Delta_{h}\theta^{1}\|.
    \end{aligned}
\end{equation}
Using the definition $\||\cdot\||_{1}$ \eqref{18-10-22-6-1-1} (defined in Theorem \ref{12345-2}), the equation \eqref{30-11-22-13} is rewritten as 
\begin{equation}\label{30-11-22-14}
\begin{aligned}
    \||\theta^{1}\||_{1}^{2} &\lesssim \tau_{1}^{\alpha}\|\nabla \mathbb{T}^{1}\|~\||\theta^{1}\||_{1}+\tau_{1}^{\alpha}\|\nabla P_{h} ~^{C}D^{\alpha}_{t_{1}}\rho\|~\|| \theta^{1}\||_{1}+\tau_{1}^{\alpha/2}\|\mathbb{Q}^{1}\|~\||\theta^{1}\||_{1}\\
    &+\tau_{1}^{\alpha/2}\|\nabla \rho^{1}\|~\||\theta^{1}\||_{1}+\tau_{1}^{\alpha/2}\|\nabla \theta^{1}\|~\||\theta^{1}\||_{1}+\tau_{1}^{1+\alpha/2}\|\Delta_{h}^{b_{2}}~\theta^{1}\|~\||\theta^{1}\||_{1}\\
    &+\tau_{1}^{1+\alpha/2}\|\nabla \theta^{1}\|~\||\theta^{1}\||_{1}+\tau_{1}^{1+\alpha/2}\|\theta^{1}\|~\||\theta^{1}\||_{1}.
    \end{aligned}
\end{equation}
Cancellation of $\||\theta^{1}\||_{1}$ from both sides of \eqref{30-11-22-14}and estimates \eqref{14-9-22-4-2}, \eqref{Sw}, \eqref{B10}, \eqref{p10}, \eqref{A10}, \eqref{30-11-22-7}, \eqref{16-10-22-5}  lead to 
\begin{equation}\label{30-11-22-15}
\begin{aligned}
    \||\theta^{1}\||_{1} &\lesssim \tau_{1}^{\alpha}t_{1}^{-\alpha}N^{-\delta \alpha}+\tau_{1}^{\alpha}h+\tau_{1}^{\alpha/2}N^{-\delta (1+\alpha)}+\tau_{1}^{\alpha/2}h+h+N^{-\delta \alpha}+\tau_{1}h+\tau_{1}N^{-\delta \alpha}\\
    &+\tau_{1}^{1}(h+N^{-\delta \alpha}).
    \end{aligned}
\end{equation}
This implies 
\begin{equation}\label{30-11-22-15-1}
    \||\theta^{1}\||_{1} \lesssim h+N^{-\delta \alpha}.
\end{equation}
Proceed further as we prove estimates \eqref{30-11-22-15-1}, \eqref{lk} and \eqref{grad bound} to complete the proof of this theorem. 
\end{proof}

\section{Numerical experiments}\label{15-10-22-4}
In this section, we implement two  numerical examples to validate the convergence results \eqref{lk} and \eqref{lks}. In our experiments, we partition the domain $\Omega$ into a uniform triangulation with $P+1$ nodes in each spatial direction and divide the time interval $[0,T]$ using graded mesh  with $N+1$ points. At first time level $t_{1}$, the numerical scheme \eqref{SUM1} gives rise to a system of nonlinear algebraic equations with non-sparse Jacobian \cite{gudi2012finite}. This system of nonlinear equations is solved by applying modified version of Newton-Raphson method \cite{gudi2012finite} with tolerance $10^{-7}$ as a stopping criterion. The value of $u_{h}^{1}$ is used to calculate $u_{h}^{2}$ and so on in the linearized numerical scheme \eqref{SU1}.
\par In order to study the convergence rates, we made several runs with different combinations of the space-time meshes. The convergence rates are calculated through the following $\log$-$\log$ formula 
\begin{equation*}
   \text{ Convergence rate} :=\begin{cases}\frac{\log(E(P_{1},N)/E(P_{2},N))}{\log(P_{1}/P_{2})}& \text{in space direction},\\
   \frac{\log(E(P,N_{1})/E(P,N_{2}))}{\log(N_{1}/N_{2})} & \text{in time direction },
\end{cases}
\end{equation*}
where $E(P_{1},N),E(P_{2},N),E(M,N_{1})$, and $E(M,N_{2})$ are errors  at different mesh points.
\begin{exam}\label{30-11-22-1} Consider the equation \eqref{16-9-22-1} in two dimensional domain  $\Omega=[0,1]\times[0,1]$ and $T=1$ such that $(x,y)\in \Omega $ and $t,s \in [0,T]$ with the following data 
 \begin{enumerate}
\item The nonlocal diffusion coefficient $M(\|\nabla u\|^{2})=1+\|\nabla u\|^{2}$ \cite{gudi2012finite, kundu2016kirchhoff}.
\item The memory operator with  coefficients \cite{ferreira2007memory}
\begin{equation*}
\begin{aligned}
    b_{2}(x,y,t,s)&=I,~ \text{where} ~I~ \text{ is the identity matrix},\\
b_{1}(x,y,t,s)&=b_{0}(x,y,t,s)=0.
    \end{aligned}
\end{equation*} 
\item Initial condition $u_{0}=0$ and source function $f(x,y,t)$ is taken  in such a way that the exact solution $u$ is given by  $u=t^{\alpha}(x-1)(y-1)\sin \pi x \sin \pi y $. This solution $u$ satisfies the regularity assumption \eqref{14-9-22-1}-\eqref{14-9-22-2}.
\end{enumerate}
\end{exam}
\noindent We observe the spatial convergence rates  by running the code at different points (P) in the space direction. We fix the number of points in the time direction as $N=5000$ so that temporal error is negligible as compared to the spatial error.
\begin{table}[htbp]
	\centering
	\caption{\emph{ Errors and convergence rates in space direction for $\alpha=0.5$}}
 	\begin{tabular}{|c|c|c|c|c|}
 	\hline
	$\delta$  & P  &	$L^{2}$ Error~~~~   Rate    &	$H^{1}_{0}$ Error~~~~   Rate      \\
 	\hline
              & 9 & 6.4781e-03~~~~~~~~~~~~~	 & 1.1817e-01~~~~~~~~~~~~~	 \\
 		
      1     & 10 & 5.2692e-03~~~~~1.9605	 & 1.0642e-01~~~~~0.9941  \\
		
           & 11 & 4.3683e-03~~~~~1.9672	 & 9.6786e-02~~~~~0.9952  \\
           
          & 12 & 3.6794e-03~~~~~1.9723	 & 8.8752e-02~~~~~0.9960  \\
    \hline
  $\simeq$ &  & ~~~~~~~~~~~~~~~~\textbf{2}   &  ~~~~~~~~~~~~~~~~\textbf{1}  \\
 \hline		
          & 9 & 6.4778e-03~~~~~~~~~~~~~	 & 1.1817e-01~~~~~~~~~~~~~	 \\
 		
      $\frac{2-\alpha}{\alpha}$    & 10 & 5.2688e-03~~~~~1.9606	 & 1.0642e-01~~~~~0.9941  \\
		
         & 11 & 4.3680e-03~~~~~1.9673	 & 9.6786e-02~~~~~0.9952  \\
           
          & 12 & 3.6791e-03~~~~~1.9725	 & 8.8752e-02~~~~~0.9960  \\
    \hline
  $\simeq$  &  &  ~~~~~~~~~~~~~~~~~\textbf{2}   &  ~~~~~~~~~~~~~~~~\textbf{1}  \\
 	\hline
 	\end{tabular}%
 	\label{table1}%
 \end{table}%
 \par   From  Table \ref{table1}, we observe that the $L^{\infty}(0,T;H^{1}_{0}(\Omega))$ norm convergence rate is linear  that confirms the result \eqref{lks}. It is interesting to note the second-order accuracy in the  $L^{\infty}(0,T;L^{2}(\Omega))$ norm, but theoretically we have only a linear order accuracy  rate in this norm \eqref{lk}. To support this  claim, we don't have enough theory yet, that will be taken up in  future. This type of phenomenon was also observed  in the literature  for Kirchhoff type problems \cite{gudi2012finite,kundu2016kirchhoff}.
\par Now, we test the accuracy of the proposed numerical scheme \eqref{SU1}-\eqref{SUM3} in the temporal direction for various values of fractional derivative exponent $\alpha$ and grading parameter $\delta$ in Table \ref{table2} and Table \ref{table3}. 
\begin{table}[htbp]
	\centering
	\caption{\emph{ $L^{\infty}(0,T;L^{2}(\Omega))$ Errors and convergence rates in time  direction }}
 	\begin{tabular}{|c|c|c|c|c|c|}
 	\hline
              &    &  $\alpha=0.4$	         &    $\alpha=0.6$        &     $\alpha=0.8$\\
 	\hline
	$\delta$  & P  &	 Error~~~~~~~~~~   Rate    &	Error~~~~~~~~~~   Rate      &  Error~~~~~~~~~~    Rate  \\
 	\hline
              & 9 & 6.4428e-03~~~~~~~~~~~~~	 & 6.5082e-03~~~~~~~~~~~~~~&  6.5575e-03~~~~~~~~~~~~~~\\
 		
      1     & 10 & 5.2398e-03~~~~~0.3922	 & 5.2934e-03~~~~~0.5882  &  5.3331e-03~~~~~0.7868\\
		
           & 11 & 4.3438e-03~~~~~0.3936	 & 4.3883e-03~~~~~0.5899  &  4.4210e-03~~~~~0.7874\\
           
          & 12 & 3.6586e-03~~~~~0.3945	 & 3.6962e-03~~~~~0.5918  &  3.7236e-03~~~~~0.7925\\
    \hline
   &  & $\min\{\delta \alpha,2-\alpha\}\simeq$~\textbf{0.4}   &$\min\{\delta \alpha,2-\alpha\}\simeq$~ \textbf{0.6}  & $\min\{\delta \alpha,2-\alpha\}\simeq$~\textbf{0.8}\\
 \hline		
          & 9 & 6.4454e-03~~~~~~~~~~~~~	 & 6.5072e-03~~~~~~~~~~~~~~&  6.5524e-03~~~~~~~~~~~~~~\\
 		
      $\frac{2-\alpha}{\alpha}$     & 10 & 5.2421e-03~~~~~1.5687	 & 5.2928e-03~~~~~1.3719  &  5.3296e-03~~~~~1.1723\\
		
           & 11 & 4.3461e-03~~~~~1.5769	 & 4.3879e-03~~~~~1.3759  &  4.4183e-03~~~~~1.1794\\
           
          & 12 & 3.6599e-03~~~~~1.5842	 & 3.6959e-03~~~~~1.3815  &  3.7215e-03~~~~~1.1958\\
    \hline
   $\simeq$ &  &$\min\{\delta \alpha,2-\alpha\}\simeq$~ \textbf{1.6}  & $\min\{\delta \alpha,2-\alpha\}\simeq$~\textbf{1.4}  & ~~$\min\{\delta \alpha,2-\alpha\}\simeq$~\textbf{1.2}\\
 	\hline
 	\end{tabular}%
 	\label{table2}%
 \end{table}%
   \par In Table \ref{table2}, errors and convergence rates in $L^{\infty}(0,T;L^{2}(\Omega))$ norm are calculated.  For $\delta =1$ ( i.e., uniform mesh in time), we set $N\simeq \lfloor P^{\frac{2}{\alpha}} \rfloor$ and notice  convergence rate of $O(N^{-\alpha})$ as  proved in \eqref{lk}. This convergence rate is not optimal, it happens because the solution $u$ has a weak singularity near $t=0$. For such type of solutions, a graded mesh which is concentrated near $t=0$ improves the accuracy rate and numerical precision of the approximate solution. Theorem \ref{Main Theorem}  and Theorem \ref{Main Theorem-1} suggest that one can achieve the best possible temporal convergence rate of  $O(N^{-(2-\alpha)})$ by choosing grading exponent $\delta \geq \frac{2-\alpha}{\alpha}$.  We remark that one should not choose larger value of $\delta$ unnecessary,  in this situation grid points more  tightly near $t=0$ and large mesh width near $t=T$ which reduces the numerical resolution of the approximate solution. Also, the constant factor involved  in Theorem \ref{Main Theorem} and Theorem \ref{Main Theorem-1}  grows with  $\delta$. Hence $\delta=\frac{2-\alpha}{\alpha}$ is the optimal choice of grading parameter in \eqref{SU4} which provides the best possible rate of convergence in the time direction. For this choice of $\delta$, we take $N\simeq \lfloor P^{\frac{2}{2-\alpha}}\rfloor$. From Table \ref{table2}, we observe that the temporal convergence rate is of $O(N^{-(2-\alpha)})$ as predicted in \eqref{lk}.
   \begin{table}[htbp]
	\centering
	\caption{\emph{ $L^{\infty}(0,T;H^{1}_{0}(\Omega))$ Errors and convergence rates in Time  direction }}
 	\begin{tabular}{|c|c|c|c|c|c|c|}
 	\hline
              &    &   & $\alpha=0.4$	         &    $\alpha=0.6$        &     $\alpha=0.8$\\
 	\hline
	$\delta$  & P  & Error	&  Rate    &	  Rate      &    Rate  \\
 	\hline
              & 9 & 1.1817e-01 &~~~~~ &~~~~~  &  ~~~~~\\
 		
      1     & 10 & 1.0642e-01 & 0.3978	 & 0.5941  & 0.7928\\
		
           & 11 & 9.6786e-02 & 0.3980	 & 0.5963  & 0.7958\\
           
          & 12 & 8.8752e-02 & 0.3984	 & 0.6035  & 0.7970\\
    \hline
   &  & $\min\{\delta \alpha,2-\alpha\}$  &$\simeq$~\textbf{0.4}   &$\simeq$~ \textbf{0.6}  & $\simeq$~\textbf{0.8}\\
 \hline
              & 9 & 1.1817e-01 &~~~~~ &~~~~~  &  ~~~~~\\
 		
      $\frac{2-\alpha}{\alpha}$     & 10 & 1.0642e-01 & 1.5861	 & 1.3948  & 1.1884\\
		
           & 11 & 9.6786e-02 & 1.5883	 & 1.4031  & 1.1961\\
           
          & 12 & 8.8752e-02 & 1.5976	 & 1.4160  & 1.2053\\
    \hline
   &  & $\min\{\delta \alpha,2-\alpha\}$  &$\simeq$~\textbf{1.6}   &$\simeq$~ \textbf{1.4}  & $\simeq$~\textbf{1.2}\\
 	\hline
 	\end{tabular}%
 	\label{table3}%
 \end{table}%
 \par In Table \ref{table3}, we compute errors and convergence rates in $L^{\infty}(0,T;H^{1}_{0}(\Omega))$ norm. For $\delta =1$, we choose $N\simeq \lfloor P^{\frac{1}{\alpha}} \rfloor$ and obtain numerical results for various values of $\alpha$. These results reflect the convergence rate  of order $\alpha$ that justify the theoretical claim \eqref{lks}. For $\delta = \frac{2-\alpha}{\alpha}$, we put $N\simeq \lfloor  P^{\frac{1}{2-\alpha}}\rfloor$ and achieve the accuracy rate of order $2-\alpha$. These convergence rates agree with the theoretical result proved in \eqref{lks}.
\par Now, we implement an another example with variable coefficients in memory operator.
\begin{exam} Consider an initial-boundary value problem \eqref{16-9-22-1} in $ \Omega=(0,1) \times (0,1)$ and $T=1$ such that $(x,y)\in \Omega $ and $t,s \in [0,T]$ with the following information 
\begin{enumerate}
\item The nonlocal diffusion coefficient $M(\|\nabla u\|^{2})=1+\|\nabla u\|^{2}$ \cite{gudi2012finite, kundu2016kirchhoff}.
\item  The memory operator with  coefficients 
\begin{equation*}
\begin{aligned}
    b_{2}(x,y,t,s)&= (1+t)(1+s)\begin{bmatrix}
1+x & 0\\
0 & 1+y 
\end{bmatrix},\\
b_{1}(x,y,t,s)&=(1+t)(1+s)\begin{bmatrix}
x \\
y 
\end{bmatrix},\\
b_{0}(x,y,t,s)&=(1+t)(1+s)xy.
    \end{aligned}
\end{equation*} 

\item Initial condition $u_{0}=0$ and  source function $f$ is chosen in such a way that the exact solution of the problem \eqref{16-9-22-1} is given by 
\begin{equation}
    u(x,y,t)=t^{\alpha}(x-x^{2})(y-y^{2})
\end{equation}
 which satisfies the regularity estimates \eqref{14-9-22-1}-\eqref{14-9-22-2}.
\end{enumerate}
\end{exam}
\noindent In this example also, we first evaluate errors and convergence rates in the space direction. For that take $N=5000$ so that spatial error dominates the convergence rates. From the previous Example \ref{30-11-22-1}, we interpret that the grading parameter $\delta = \frac{2-\alpha}{\alpha}$ gives the better convergence rates. Therefore, in this example we collect the numerical results for grading parameter $\delta = \frac{2-\alpha}{\alpha}$. From the Table \ref{table4}, we see that the convergence rate is quadratic in $L^{\infty}(0,T;L^{2}(\Omega))$ norm and linear in $L^{\infty}(0,T;H^{1}_{0}(\Omega))$ norm. 
\begin{table}[htbp]
	\centering
	\caption{\emph{ Errors and convergence rates in space direction for $\alpha=0.5$ and $\delta=\frac{2-\alpha}{\alpha}$}}
 	\begin{tabular}{|c|c|c|c|}
 	\hline
	                 P &	  $L^{2}$ Error~~~~   Rate    &	$H^{1}_{0}$ Error~~~~   Rate      \\
 \hline		
                             9 &     8.9768e-04~~~~~~~~~~~	      & 2.0361e-02~~~~~~~~~~~~  \\
 		
   10 &    7.2881e-04~~~~~1.9780	      & 1.8257e-02~~~~~1.0355  \\
		
                             11 &    6.0337e-04~~~~~1.9818	      & 1.6549e-02~~~~~1.0305  \\
           
                             12 &    5.0767e-04~~~~~1.9847	      & 1.5135e-02~~~~~1.0265  \\
    \hline
                  $\simeq$      &  ~~~~~~~~~~~~~~~~~\textbf{2}   &  ~~~~~~~~~~~~~~~~\textbf{1}  \\
 	\hline
 	\end{tabular}%
 	\label{table4}%
 \end{table}%
\par  Further, we investigate the convergence rates in the time direction in $L^{\infty}(0,T;L^{2}(\Omega))$ norm as well as in $L^{\infty}(0,T;H^{1}_{0}(\Omega))$ norm. In Table \ref{table5}, we get the numerical results for $L^{\infty}(0,T;L^{2}(\Omega))$ norm with $\delta = \frac{2-\alpha}{\alpha}$ by putting  $N\simeq \lfloor P^{\frac{2}{2-\alpha}}\rfloor$. From these results we can say that the convergence rate is  of $O(N^{-(2-\alpha)})$ which is in accordance with the theoretical outcome \eqref{lk}.
 \begin{table}[htbp]
	\centering
	\caption{\emph{ $L^{\infty}(0,T;L^{2}(\Omega))$ Errors and convergence rates in time  direction for $\delta = \frac{2-\alpha}{\alpha}$  }}
 	\begin{tabular}{|c|c|c|c|c|c|}
 	\hline
                                       &  $\alpha=0.2$   &  $\alpha=0.4$	         &    $\alpha=0.6$        &     $\alpha=0.8$\\
 	\hline
	                                 P  &	  Error  ~~~~~ Rate    &  Error   Rate    &	Error  Rate      &  Error Rate  \\
 		\hline
                                     9 & 1.9132e-03~~~~~~~~~~	 & 1.0081e-03~~~~~~~~~~	 & 8.6619e-04~~~~~~~~~~ &  8.2069e-04~~~~~~~~~~~\\
 		
                                     10 & 1.5653e-03~~1.7811	 & 8.2233e-04~~1.5468	 & 7.0303e-04~~1.3861  &  6.6538e-04~~1.1906\\
		
                                     11 & 1.2958e-03~~1.7856	 & 6.8299e-04~~1.5619	 & 5.8189e-04~~1.3878  &  5.5028e-04~~1.1947\\
           
                                     12 & 1.0918e-03~~1.7994	 & 5.7596e-04~~1.5714	 & 4.8951e-04~~1.3916  &  4.6264e-04~~1.2087\\
    \hline
                                     & $\simeq$ ~~~~~~~~ \textbf{1.8} &$\simeq$  ~~~~~~~~\textbf{1.6}  & $\simeq$ ~~~~~~~~\textbf{1.4}  &$\simeq$ ~~~~~~~~\textbf{1.2}\\
 	\hline
 	\end{tabular}%
 	\label{table5}%
 \end{table}%
\par  Finally, $L^{\infty}(0,T;H^{1}_{0}(\Omega))$ norm errors and convergence rates are evaluated in Table \ref{table6} for grading parameter $\delta = \frac{2-\alpha}{\alpha}$ by taking $N\simeq \lfloor P^{\frac{1}{2-\alpha}}\rfloor$. We achieve the convergence rate of order $2-\alpha$ in $L^{\infty}(0,T;H^{1}_{0}(\Omega))$ norm which is same as established in estimate \eqref{lks}.
 \begin{table}[htbp]
	\centering
	\caption{\emph{ $L^{\infty}(0,T;H^{1}_{0}(\Omega))$ Errors and convergence rates in time  direction for $\delta = \frac{2-\alpha}{\alpha}$  }}
 	\begin{tabular}{|c|c|c|c|c|c|}
 	\hline
                                       &  $\alpha=0.2$   &  $\alpha=0.4$	         &    $\alpha=0.6$        &     $\alpha=0.8$\\
 	\hline
	                                 P  &	  Error  ~~~~~ Rate    &  Error   Rate    &	Error  Rate      &  Error Rate  \\
 		\hline
                                     9 & 2.3239e-02~~~~~~~~~~	 & 2.0551e-02~~~~~~~~~~	 & 2.0246e-02~~~~~~~~~~ &  2.0105e-02~~~~~~~~~~\\
 		
                                     10 & 2.0924e-02~~1.8821	 & 1.8423e-02~~1.7482	 & 1.8170e-02~~1.4394  &  1.8063e-02~~1.2256\\
		
                                     11 & 1.8995e-02~~1.8323	 & 1.6695e-02~~1.6746	 & 1.6482e-02~~1.4316  &  1.6399e-02~~1.2233\\
           
                                     12 & 1.7360e-02~~1.7944	 & 1.5263e-02~~1.6139	 & 1.5082e-02~~1.4297  &  1.5017e-02~~1.2026\\
    \hline
                                     & $\simeq$ ~~~~~~~~ \textbf{1.8} &$\simeq$  ~~~~~~~~\textbf{1.6}  & $\simeq$ ~~~~~~~~\textbf{1.4}  &$\simeq$ ~~~~~~~~\textbf{1.2}\\
 	\hline
 	\end{tabular}%
 	\label{table6}%
 \end{table}%

 \begin{rmk}
 From Tables \ref{table5} and \ref{table6}, we remark that by choosing $\delta=\frac{2-\alpha}{\alpha}$ the convergence rate is of order $2-\alpha$ for smaller values of $\alpha$ as well as larger value of $\alpha$. However, these graded mesh has its own defect. For example, if $\alpha$ is small then grading exponent $\delta$ becomes large that causes faster growth of the constant involved in Theorem \ref{Main Theorem} and Theorem \ref{Main Theorem-1}. This phenomena reduces the numerical resolution of the approximate solution and produces larger error as in Table \ref{table5}.
 \end{rmk}

\section{Concluding Remarks}\label{15-10-22-5}
In this work, we  introduced a new type of Ritz-Volterra projection operator using  which  semi discrete error analysis for the equation $(\mathcal{D}^{\alpha})$ has been established. In comparison to the Newton-Raphson method,  we have successfully reduced the computational cost and storage by constructing a novel linearization algorithm \eqref{SU1}-\eqref{SUM3} for the equation  $(\mathcal{D}^{\alpha})$. We presented a new approach to derive a priori bounds on the numerical solution and  global convergence rate for the proposed fully discrete scheme. Finally, numerical experiments revealed  the robustness and efficiency of the developed numerical scheme.

\bibliographystyle{plain}
\bibliography{FP3BIB}

\end{document}